\documentclass{article}

\usepackage{cite}
\usepackage{amssymb}
\usepackage{amsmath}
\usepackage{amsfonts}
\usepackage{amsthm}
\usepackage{enumerate}
\usepackage{colonequals}
\usepackage{mathrsfs} 
\usepackage{color}
\usepackage{enumitem}
\usepackage{mathtools}
\usepackage[all,cmtip]{xy}
\usepackage{tikz-cd} 
\usepackage{multicol}
\usepackage{float}
\usepackage{placeins}
\usepackage{accents}

\usetikzlibrary{decorations.markings}

\usepackage{url}

\usepackage{euscript}
\usepackage{wasysym}

\newtheorem{theorem}{Theorem}[section]
\newtheorem{claim}{Claim}[section]
\newtheorem{proposition}{Proposition}[section]

\newtheorem*{theorem*}{Theorem}
\newtheorem{corollary}{Corollary}[section]
\newtheorem{definition}{Definition}[section]

\newtheorem{lemma}{Lemma}[section]

\newtheorem{remark}{Remark}[section]

\newtheorem{question}{Question}[section]

\newcommand{\eusc}[1]{\EuScript{#1}}

\newcommand{\mtt}[1]{\mathtt{#1}}

\numberwithin{equation}{section}

\allowdisplaybreaks[2]

\newcommand{\deee}{\hspace{2 pt} \mathrm{d}}

\newcommand{\ov}[1]{\overline{#1}}

\newcommand{\mrm}[1]{\mathrm{#1}}
\newcommand{\nml}{\left \vert \left \vert}
\newcommand{\nmr}{\right \vert \right \vert}
\newcommand{\nmbl}{\bigl \vert \bigl \vert}
\newcommand{\nmbr}{\bigr \vert \bigr \vert}

\newcommand{\qqquad}{\qquad \qquad}
\renewcommand{\hat}{\widehat}

\DeclareMathOperator{\tr}{\mathrm{tr}}

\DeclareMathOperator{\dm}{\mathrm{dm}}

\DeclareMathOperator{\spn}{\mathrm{spn}}

\DeclareMathOperator{\dss}{\mtt{DS}}

\newtheorem*{alt_theorem}{Theorem A}

  \usepackage{stackengine,scalerel}

\allowdisplaybreaks

\begin{document}

\title{Hyperlinear approximations to amenable groups come from sofic approximations}
\author{Peter Burton, Maksym Chaudkhari, Kate Juschenko, Kyrylo Muliarchyk}

\maketitle

\begin{abstract} We provide a quantitative formulation of the equivalence between hyperlinearity and soficity for amenable groups, effectively showing how every hyperlinear approximation to such a group is simulated by a suitable sofic approximation. The proof is probabilistic, using the concentration of measure in high-dimensional spheres to control the deviation of an operator's matrix coefficients from its trace. As a corollary, we obtain a result connecting stability of sofic approximations with stability of hyperlinear approximations.

\end{abstract}

\tableofcontents

\section{Introduction}

\subsection{Hyperlinear and sofic groups}

Hyperlinear and sofic groups were introduced independently of each other and come from different corners of mathematics: operator algebras and symbolic dynamics. Hyperlinear groups have their origin in Connes’ embedding conjecture from \cite{MR0454659} about von Neumann factors of type $\mrm{II}_1$. Sofic groups were introduced by Gromov in \cite{Gromov:1999aa}, motivated by Gottschalk's surjunctivity conjecture. Both notions reflect an idea of approximating infinite groups by finite or finite dimensional structures.\\

In the sofic case, the relevant approximations are partial actions by permutations of finite sets which model the action of the group on itself by left-translations. In the hyperlinear case, the relevant approximations are partial unitary representations on finite dimensional Hilbert spaces which model the left regular representation of the group. It is unknown whether every countable discrete group is sofic, and the same question is open for hyperlinearity. We refer the reader to the references \cite{MR3408561}  or \cite{MR2460675}  for surveys on these topics.\\
\\
We now present the relevant definitions. Turning first to soficitiy, we write $\mathrm{Sym}(V)$ for the group of permutations of a finite set $V$.

\begin{definition} Let $G$ be a countable discrete group, let $F \subseteq G$ be finite and let $\epsilon > 0$. We define an $(F,\epsilon)$\textbf{-sofic approximation} to $G$ to consist of a finite set $V$ and a map $\sigma:G \to \mathrm{Sym}(V)$ which obeys the following conditions.
 
 \begin{itemize} \item For all $g,h \in F$ we have: \[ |\{v \in V: \sigma(g)\sigma(h)v \neq \sigma(gh)v\}| \leq \epsilon|V|  \] \item For every distinct pair $g,h \in F$ we have: \[ |\{v \in V:\sigma(g)v = \sigma(h)v\}| \leq \epsilon|V| \] \end{itemize} We define a \textbf{sofic approximation sequence to }$G$ to consist of a finite set $V_n$ and an $(F_n,\epsilon_n)$-sofic approximation $\sigma_n:G \to \mrm{Sym}(V_n)$ for each $n \in \mathbb{N}$, where $(F_n)_{n \in \mathbb{N}}$ is an increasing sequence of finite subsets of $G$ whose union is the entire group and $(\epsilon_n)_{n \in \mathbb{N}}$ is a decreasing sequence of positive numbers whose limit is zero. We also define the group $G$ to be \textbf{sofic} if there exists a sofic approximation sequence to $G$. \end{definition}

We now define hyperlinear approximations. If $\mathcal{H}$ is a finite-dimensional Hilbert space and $s$ is a linear operator on $\mathcal{H}$ we write $\tr(s)$ for the trace of $s$. We also define the \textbf{Hilbert-Schmidt norm} of $s$ by $||s||_{\mrm{HS}} = \sqrt{\tr(s^\ast s)}$, and write $\mrm{Un}(\mathcal{H})$ for the group of unitary operators on $\mathcal{H}$.

\begin{definition} \label{def.hyper} Let $G$ be a countable discrete group, let $F$ be a finite subset of $G$ and let $\epsilon > 0$. We define an $(F,\epsilon)$\textbf{-hyperlinear approximation to }$G$ to consist of a finite dimensional Hilbert space $\mathcal{H}$ along with a map $\alpha:G \to \mrm{Un}(\mathcal{H})$ which obeys the following conditions. \begin{itemize} \item For all $g,h \in F$ we have $\|\alpha(gh)-\alpha(g)\alpha(h)\|_{\mrm{HS}}^2 \leq \epsilon \dm(\mathcal{H})$. \item For all distinct pairs $g,h \in F$ we have $|\tr(\alpha(h^{-1}g))| \leq \epsilon \dm(\mathcal{H})$. \end{itemize} We define a \textbf{hyperlinear approximation sequence to }$G$ to consist of a finite dimensional Hilbert space $\mathcal{H}_n$ and an $(F_n,\epsilon_n)$-hyperlinear approximation $\alpha_n:G \to \mrm{Un}(\mathcal{H}_n)$ for each $n \in \mathbb{N}$, where $(F_n)_{n \in \mathbb{N}}$ is an increasing sequence of finite subsets of $G$ whose union is the entire group and $(\epsilon_n)_{n \in \mathbb{N}}$ is a decreasing sequence of positive numbers whose limit is zero.  We also define the group $G$ to be \textbf{hyperlinear} if there exists a hyperlinear approximation sequence to $G$.  \end{definition}
 
Observe that if $\sigma:G \to \mrm{Sym}(V)$ is an $(F,\epsilon)$-sofic approximation to $G$ then the corresponding permutation matrices acting on $\ell^2(V)$ form an $(F,\epsilon)$-hyperlinear approximation to $G$. We will refer to a hyperlinear approximation produced in this way as \textbf{sofic-induced}. It follows directly from this observation that every sofic group is hyperlinear, but a fundamental open question in the theory of finite approximations to infinite groups is the following.

\begin{question} \label{que.hypsof} Is every hyperlinear group sofic? \end{question}

We note that in the literature on hyperlinearity, the second condition in Definition \ref{def.hyper} is often replaced by a condition of the form $||\alpha(g) - \alpha(h)||_{\mrm{HS}} \geq c \sqrt{\dm(\mathcal{H})}$ for an absolute constant $c$ with $0 < c < \sqrt{2}$. For example, the reference \cite{MR2460675} takes $c = 1/4$. These conditions for various choices of $c$ all provide equivalent definitions of a hyperlinear \textit{group}, and our definition also gives the same class of groups. However, in studying particular hyperlinear approximations to a fixed group (as opposed to hyperlinearity as a group property) it becomes clear that our condition on small traces is more natural than assuming a lower bound on the Hilbert-Schmidt distance. The essential reason is that bounds of the form \begin{equation} \label{eq.real} (\sqrt{2}+\epsilon) \sqrt{\dm(\mathcal{H})} \geq ||\alpha(g) - \alpha(h)||_{\mrm{HS}} \geq (\sqrt{2}-\epsilon) \sqrt{\dm(\mathcal{H})} \end{equation} for some small $\epsilon > 0$ will guarantee that the real part of the trace of $\alpha(h^{-1}g)$ is small, but provide no information about the imaginary part of this trace. The philosophy of understanding a hyperlinear approximation as an approximate embedding of the group von Neumann algebra into the hyperfinite $\mrm{II}_1$ factor requires that the trace itself is small, and so this is the approach we take. It is also a closer match to the definition of soficity, where the trace coincides with the fixed-point count.

\subsection{Statement and discussion of main result} \label{sec.stat}

An approach that has proved profitable in studying whether all countable groups are hyperlinear or sofic is as follows. One starts by identifying some particularly nice family of approximations to the group, such as those by genuine finite dimensional representations or genuine actions on finite sets, and then attempts to understand the extent to which these examples might be in some sense exhaustive. These phenomena are generally referred to as stability properties, and numerous variants have been studied in references such as \cite{MR4027744}, \cite{MR3999445}, \cite{ioana2021almost} and \cite{levit_lubotzky_2022}. \\
\\
We take a similar perspective on Question \ref{que.hypsof}. By considering partial actions on F{\o}lner sets one can see that every amenable group is sofic, and hence every amenable group is hyperlinear. This raises the question as to whether sofic-induced hyperlinear approximations are the only way to produce hyperlinear approximations to amenable groups. The answer to ths question is positive, and the purpose of the present article is to provide an effective proof of this fact. Our main result is the following theorem.

\begin{theorem} \label{thm.sofic} Let $G$ be a countable amenable group, let $E$ be a finite subset of $G$ and let $\epsilon > 0$. Then there exists a finite subset $F$ of $G$ and $\delta > 0$ with the following property. If $\alpha:G \to \mrm{Un}(\mathcal{H})$ is an $(F,\delta)$ hyperlinear approximation to $G$ then there exists a hyperlinear approximation $\beta:G \to \mrm{Un}(\mathcal{H})$  which is induced from an $(E,\epsilon)$-sofic approximation and such that $\|\beta(g) -\alpha(g)\|_{\mrm{HS}}^2 \leq \epsilon \dm(\mathcal{H})$ for all $g \in E$. Moreover, $F$ and $\delta$ could be constructed by a recursive procedure which has the following properties:
\begin{enumerate}
    \item At each step the procedure receives a finite subset of $G$ and a positive constant, and returns a suitable F{\o}lner subset of $G$ and a new positive constant. The procedure starts with $E$ and $\epsilon$.
    \item The number of iterations depends only on $\epsilon$ and this dependence could be made explicit. 
    \item If the input at step $k$ is a subset $F_k$ and a constant $\delta_k$, then the $F_k$-invariance of the new set $F_{k+1}$ and the parameter $\delta_{k+1}$ of the output admit an effective description in terms of $|F_k|$ and $\delta_k$. Moreover, this description does not depend on the choice of $G$, as long as we assume that $G$ is torsion-free.
    
\end{enumerate}
\end{theorem}

The explicit dependence of $(F,\delta)$ on $(E,\epsilon)$ and the description of the procedure mentioned in the theorem could be read from the statement of Lemmas \ref{main.lemma} and \ref{lem.1} and the first paragraph in the proof of Theorem \ref{thm.sofic}.\\
\\
To consider an elementary but illustrative example, take $G = \mathbb{Z}$ in Theorem \ref{thm.sofic}. In this case the result asserts that if a unitary operator $u \in \mrm{Un}(\mathcal{H})$ satisfies $|\tr(u^k)| \leq \epsilon \dm(\mathcal{H})$ for some small $\epsilon > 0$ and all nonzero $k \in \{-n,\ldots,n\}$ for some large $n \in \mathbb{N}$, then $u$ must be close to a permutation matrix. This can be seen directly as follows. First observe that the Fourier transform allows us to interpret these inequalities as a weak equidistribution statement for the eigenvalues of $u$. A direct sum of cyclic permutation matrices with appropriate length will have an eigenvalue distribution that is similarly equidistributed, and so a permutation matrix close to $u$ can be produced by `rounding eigenvalues'. This Fourier-analytic approach can be extended directly to prove Theorem \ref{thm.sofic} for arbitrary abelian groups, and is incorporated into our method in Section \ref{sec.locspec} below. The example of a unitary operator with eigenvalues uniformly distributed on the upper half of the unit circle also shows that estimating the entire trace is necessary for Theorem \ref{thm.sofic}, as the estimates on the real part arising from something like (\ref{eq.real}) do not force the required equidistribution on the imaginary axis.

\subsection{Applications of main result}

\subsubsection{Stable soficity and hyperlinearity}

Theorem \ref{thm.sofic} can be interpreted as a kind of `hyperlinear-to-sofic' stability for amenable groups, and as such it has a consequence for relating two forms of the `approximate-to-exact' stability mentioned at the beginning of Section \ref{sec.stat}. We now provide some definitions needed to formulate this consequence. The following is the main topic of \cite{MR4105530}.

\begin{definition} \label{def.stable-sof} We define a countable group $G$ to be \textbf{stably sofic} if for every sofic approximation sequence $(\sigma_n: G \to \mrm{Sym}(V_n))_{n \in \mathbb{N}}$ to $G$ there exists a sequence of finite sets $(W_n)_{n \in \mathbb{N}}$ satisfying \[ \lim_{n \to \infty } \frac{|V_n \cap W_n|}{|V_n|} = \lim_{n \to \infty } \frac{|V_n \cap W_n|}{|W_n|} =1  \] along with group homomorphisms $(\tau_n:G \to \mrm{Sym}(W_n))_{n \in \mathbb{N}}$ such that the following limit holds for every $g \in G$. \[ \lim_{n \to \infty} \frac{1}{|V_n \cap W_n|} |\{v \in V_n \cap W_n:\sigma_n(g)v = \tau_n(g)v\}| = 1 \] \end{definition}

The next definition is the `hyperlinear analog' of the previous one, and was explored in \cite{2022arXiv221110492D}.

\begin{definition} \label{def.stable-hyp} We define a countable group $G$ to be \textbf{stably hyperlinear} if for every hyperlinear approximation sequence $(\sigma_n: G \to \mrm{Un}(\mathcal{H}_n))_{n \in \mathbb{N}}$ to $G$ there exists a sequence of finite dimensional Hilbert spaces $(\mathcal{K}_n)_{n \in \mathbb{N}}$ satisfying \[ \lim_{n \to \infty } \frac{\dm(\mathcal{H}_n \cap \mathcal{K}_n)}{\dm(\mathcal{H}_n)} = \lim_{n \to \infty }  \frac{\dm(\mathcal{H}_n \cap \mathcal{K}_n)}{\dm(\mathcal{K}_n)} =1  \] along with unitary representations $(\beta_n:G \to \mrm{Un}(\mathcal{K}_n))_{n \in \mathbb{N}}$ such that the following limit holds for every $g \in G$, where $p_n:\mathcal{H}_n \twoheadrightarrow \mathcal{H}_n \cap \mathcal{K}_n$ and $q_n:\mathcal{K}_n \twoheadrightarrow \mathcal{H}_n \cap \mathcal{K}_n$ denote the respective orthogonal projections. \[ \lim_{n \to \infty} \frac{||p_n\alpha(g)p_n - q_n \beta(g)q_n||_{\mrm{HS}}}{\sqrt{\dm(\mathcal{H}_n \cap \mathcal{K}_n)}}  = 0 \] \end{definition}

With these definitions, we have the following immediate corollary of Theorem \ref{thm.sofic}.

\begin{corollary} \label{thm.2} Let $G$ be a countable discrete amenable group and suppose $G$ is stably sofic. Then $G$ is stably hyperlinear. \end{corollary}

\subsubsection{Free products with amalgamation}

In \cite{MR2869256}, the authors state that the following theorem is a consequence of the work in \cite{MR2439665}.

\begin{theorem} \label{thm.amalg} Let $G$ and $H$ be hyperlinear groups, and let $K$ be a common amenable subgroup of $G$ and $H$. Then the free product of $G$ and $H$ amalgamated over $K$ is again hyperlinear.\end{theorem}

The required results in \cite{MR2439665} rely on the fact that a group $G$ is amenable if and only if its group von Neumann algebra is approximately finite dimensional. This theorem ultimately goes back to Connes' classification of injective factors from \cite{MR0454659}. While subsequent efforts such as \cite{MR2168856}  and \cite{MR0860346} have simplified the proof somewhat, it is still a nonconstructive argument involving abstract von Neumann algebras. On the hand, it is worth mentioning that this fact allows one to obtain a non-effective version of our Theorem \ref{thm.sofic}. Indeed, it implies that for a fixed ultraproduct $\mathcal{U}$ of unitary groups and any amenable group $G$ there is a unique up to conjugation trace-preserving embedding of $G$ into $\mathcal{U}$. Moreover, for any hyperlinear approximation of sufficently high dimension one can find a sofic-induced hyperlinear approximation of the same dimension. Then the corresponding embeddings of $G$ in the ultraproduct of unitary groups are conjugate, and this implies that hyperlinear approximations are almost conjugate to sofic approximations.  However, this folklore argument does not provide any effective description of a suitable set $F$ in Theorem \ref{thm.sofic} \\
\\
In \cite{MR2823074}, Elek and Szab\'{o} show that the free product of two sofic groups amalgamated over a common amenable subgroup is again sofic. In contrast with the theory discussed in the previous paragraph, their proof is quite direct, producing a sofic approximation to the amalgamated free product as a sort of fibration over a randomly constructed sofic approximation to the amenable subgroup. Using our Theorem \ref{thm.sofic}, their method provides an similarly direct proof of Theorem \ref{thm.amalg}, which we now sketch.\\
\\
Essentially the only modification needed is to `linearize' the notion of enhancement appearing in Definition 3.4 of \cite{MR2823074}. In this definition, Elek and Szab\'{o} consider enriching a sofic approximation to the amenable subgroup $K$ with the structure of a sofic approximation to the larger groups $G$ and $H$. Starting from arbitrary hyperlinear approximations to $G$ and $H$, Theorem \ref{thm.sofic} allows us to assume that the respective restrictions to $K$ are in fact sofic approximations. We can use Corollary 3.3 from \cite{MR2823074} to stitch together these sofic approximations to $K$ in a `relatively independent' way, and then enhance them with the hypothesized hyperlinear approximations to $G$ and $H$ (instead of sofic approximations to $G$ and $H$). All other aspects of the proof from \cite{MR2823074} work identically to produce the desired hyperlinear approximation to the amalgamated free product.

\subsection{Outline of proof of Theorem \ref{thm.sofic}} \label{sec.outline}

The remainder of this article is devoted to the proof of Theorem \ref{thm.sofic}, and we begin with a broad outline of the argument. Theorem \ref{thm.sofic} will follow reasonably quickly from our `main lemma', which we state as Lemma \ref{main.lemma} below. Heuristically, this result shows that for any finite subset $E$ of an amenable group $G$ and any sufficiently good hyperlinear approximation $\alpha:G \to \mrm{Un}(\mathcal{H})$, we can find an almost invariant subspace $\mathcal{V} \leq \mathcal{H}$ with $\dm(\mathcal{V})$ being a fixed proportion of $\dm(\mathcal{H})$ and such that the action of $\alpha(E)$ on $\mathcal{V}$ is approximated by a permutation action on an orthonormal basis of $\mathcal{V}$. Here and throughout the article, we adopt the convention that a projection on a Hilbert space is to be understood as an orthogonal projection - i.e. as an operator $p$ satisfying $p = p^2 = p^\ast$. We also write $\mrm{Proj}(\mathcal{H})$ for the set of all projections on $\mathcal{H}$ and $I_\mathcal{H}$ for the identity operator on $\mathcal{H}$.

\begin{lemma}[Main lemma] \label{main.lemma}
For all  amenable groups $G$, all finite subsets $E$ of $G$ and all $\nu, \kappa > 0$ there exist a finite subset $F$ of $G$ and $\delta > 0$ such that if $\alpha:G \to \mrm{Un}(\mathcal{H})$ is an $(F,\delta)$-hyperlinear approximation to $G$ then there exists a projection $p \in \mrm{Proj}(\mathcal{H})$ with $\tr(p) \geq \kappa^2 \dm(\mathcal{H})/2$ such that the following statement holds. There exists a hyperlinear approximation $\beta: G \to \mrm{Un}(p\mathcal{H})$ which is induced from an $(E,\nu)$-sofic approximation and an $(E,7\nu )$-hyperlinear approximation $\gamma: G \to \mrm{Un}(\mathcal{H} \ominus p\mathcal{H})$ such that for every $g \in E$ we have
\begin{equation}\label{mainlemma1}
    \|(\alpha(g)-\beta(g))p\|_{\mrm{HS}}^2\leq (5\kappa+\nu)^2\tr(p)
\end{equation}
and
\begin{equation}\label{mainlemma2}
    \|(\alpha(g)-\gamma(g))(I_\mathcal{H}-p)\|_{\mrm{HS}}^2\leq 4\nu^2(\dm(\mathcal{H}) - \tr(p))
\end{equation}
\end{lemma}

The projection $p$ in Lemma \ref{main.lemma} is constructed by a recursive process described in Section \ref{main}, for which the main ingredients are Lemmas \ref{prop.ainv} and \ref{lem.1} from Section \ref{sec.ainv}. This procedure iteratively builds an orthonormal basis for the range of $p$, enlarging it step-by-step by with collections of $|F|$ orthonormal vectors which obey  conditions sufficient for Lemma \ref{main.lemma}.  This scheme can be regarded as the `innermost' of two recursions, with the corresponding `outer recursion' involving repeated applications of Lemma \ref{main.lemma} to establish Theorem \ref{thm.sofic}. The basic idea recalls the `Rokhlin lemma for sofic approximations' which appears as Lemma 4.5 in \cite{MR3068400}, although the linear-algebraic setting presents numerous difficulties not present in the permutational setting. 

The reader might also notice that after minor modifications our proof also allows us to obtain the following version of Theorem \ref{thm.sofic}.

\begin{alt_theorem}
There exist functions $\mathfrak{T}, \mathfrak{P} :\mathbb{N} \to \mathbb{N}$ with the following property. Let $G$ be a torsion-free countable group and let $N \in \mathbb{N}$. Also let $E$ be a finite symmetric subset of $G$ and let $F_1,\ldots,F_N$ be an increasing sequence of symmetric finite subsets of $G$ such that $F_1 = E$ and the following two conditions hold for every $n \in \{1,..., N\}$.
\begin{enumerate}
    \item   $|F_n F_{n+1} \triangle F_{n+1}| \leq \dfrac{|F_{n+1}|}{\mathfrak{T}(|F_n|)} $
    \item We have $g^k \in F_{n+1}$ for every $g \in F_n$ and every $k \in \{ -\mathfrak{P}(\mathfrak{T}(|F_n|), \ldots, \mathfrak{P}(\mathfrak{T}(|F_n|) \}$ 
\end{enumerate}

Let $F=F_N$ and $\delta = 1/\mathfrak{T}(|F_N|)$. Then for any $(F,\delta)$-hyperlinear approximation $\alpha:G \to \mrm{Un}(\mathcal{H})$ there exists a sofic-induced $(E,1/N)$-hyperlinear approximation $\beta:G \to \mrm{Un}(\mathcal{H})$ satisfying $||\alpha(g) - \beta(g)||_{\mrm{HS}} \leq \sqrt{\dm(\mathcal{H})}/N$ for all $g \in E$.
\end{alt_theorem}

We now give a section-by-section summary of the remainder of the article.

\begin{itemize} \item Section \ref{sec.linalg} collects a number of results in `approximate linear algebra', mostly about how perturbations of certain constructions can be repaired to meet the exact definitions. \item Section \ref{sec.concen} introduces the fundamental connection between normalized traces and expected values of bilinear forms, along with the concentration of measure bounds needed to promote average properties to typical properties. \item Section \ref{sec.disunif} deals with a notion of `disuniformity' for probability measures on the circle, which will be used to track the extent to which eigenvalues of unitary operators in hyperlinear approximations are equidistributed. \item Section \ref{sec.locspec} provides a framework for analyzing the spectral measures of unitary operators restricted to subspaces of their domain, and in particular connects the disuniformity estimates from Section \ref{sec.disunif} with the linear-algebraic estimates from Section \ref{sec.linalg}. \item Section \ref{sec.concest} collects most of the results of the previous sections into statements asserting that vectors in the unit sphere of the underlying Hilbert space of a hyperlinear approximation typically satisfy desirable properties. \item Section \ref{sec.ainv} is devoted to the proof of Lemmas \ref{prop.ainv} and \ref{lem.1}. These results are the technical core of the argument. Roughly, they assert that if vectors are chosen so as to satisfy all the typical properties described in Section \ref{sec.concest} then their span satisfies the properties needed by the main recursive constructions. \item Section \ref{main} carries out the inner recursion to prove Lemma \ref{main.lemma}, and then the outer recursion to prove Theorem \ref{thm.sofic}. This section relies on the previous ones only via Lemmas \ref{prop.ainv} and \ref{lem.1} from Section \ref{sec.ainv}.
	
\end{itemize}

\begin{remark} As the proof of Theorem \ref{thm.sofic} is already somewhat lengthy, we choose to carry it our under the simplifying assumption that the group $G$ is torsion-free. The modifications to deal with torsion elements are essentially notational and are pointed out where they occur, in Remark \ref{rem.torsion-1} from Section \ref{sec.loccc} and Remark \ref{rem.torsion-2} from Section \ref{sec.projconc}. \end{remark}

\subsection{Acknowledgements} The authors are grateful to Andreas Thom for careful reading of the paper and for the remarks that helped us to improve the exposition.

\section{Linear algebraic estimates} \label{sec.linalg}

\subsection{Orthogonalization bound}

Given a set $S$ of vectors in a vector space, it will be convenient to use the notation $\spn(S)$ for the linear span of $S$.
 The result below is a `repair principle' for turning approximately orthonormal sets into exactly orthonormal sets. 

\begin{proposition} \label{prop.orthogonal} Let $n \in \mathbb{N}$ and let $\delta > 0$. Suppose $\xi_1,\ldots,\xi_n$ is a list of linearly independent unit vectors in a Hilbert space $\mathcal{H}$ such that $|\langle \xi_j,\xi_k \rangle| \leq \delta$ for all distinct pairs $j,k \in \{1,\ldots,n\}$. Then there exists a list $\vartheta_1,\ldots,\vartheta_n$ of orthonormal vectors in $\mathcal{H}$ such that $\|\xi_j - \vartheta_j\| \leq \delta n$ for all $j \in \{1,\ldots,n\}$ and such that $\spn(\xi_1,\ldots,\xi_n) = \spn(\vartheta_1,\ldots,\vartheta_n)$. \end{proposition}

\begin{proof}[Proof of Proposition \ref{prop.orthogonal}] By replacing $\mathcal{H}$ with $\spn(\xi_1,\ldots,\xi_n)$ if necessary, we may assume $\dm(\mathcal{H}) \leq n$. Moreover, if $\omega_1,\ldots,\omega_n$ is any list of vectors in a Hilbert space $\mathcal{E}$ of dimension $n$ such that $\langle \omega_j,\omega_k \rangle = \langle \xi_j,\xi_k \rangle$ for all $j,k \in \{1,\ldots,n\}$ then there exists a unitary isomorphism $u:\mathcal{E} \to \mathcal{H}$ such that $u \omega_k = \xi_k$ for all $k \in \{1,\ldots,n\}$. These observations imply that in order to verify the present proposition it suffices to establish the following claim.

\begin{claim} \label{cla.orthogonal} 
There exists a Hilbert space $\mathcal{E}$ of dimension $n$ and vectors $\omega_1,\ldots,\omega_k \in \mathcal{E}$ such that $\langle \omega_j, \omega_k \rangle = \langle \xi_j,\xi_k \rangle$ for all $j,k \in \{1,\ldots,n\}$ and such that for some orthonormal family $\varpi_1,\ldots,\varpi_n \in \mathcal{E}$ we have $\|\omega_k - \varpi_k\| \leq \delta n$ for all $k \in \{1,\ldots,n\}$. 
\end{claim} We will in fact obtain Claim \ref{cla.orthogonal} with $\mathcal{E} = \mathbb{C}^n$. Let $\eusc{A}$ be the $n \times n$ matrix whose $(j,k)$-entry is equal to $\langle \xi_j,\xi_k \rangle$. Since $\eusc{A}$ is positive semidefinite, according to the spectral theorem there exists an $n \times n$ unitary matrix $\eusc{U}$ and an $n \times n$ diagonal matrix $\eusc{D}$ with nonnegative entries such that $\eusc{A} = \eusc{U}^\ast\eusc{D}\eusc{U}$. Let $\varpi_1,\ldots,\varpi_n \in \mathbb{C}^n$ be the columns of $\eusc{U}$, so that $\varpi_1,\ldots,\varpi_n$ is an orthonormal set. Also let $\sqrt{\eusc{D}} \in \mathrm{End}(\mathbb{C}^n)$ be the diagonal matrix whose entries are the square roots of the entries of $\eusc{D}$ and write $\omega_k = \sqrt{\eusc{D}} \varpi_k$ for $k \in \{1,\ldots,n\}$. \\
\\
Letting $\upsilon_1,\ldots,\upsilon_n$ be the standard basis for $\mathbb{C}^n$, for all $j,k \in \{1,\ldots,n\}$ we have by construction that 
\[
\omega_j^\ast \omega_k = \varpi_j^\ast \eusc{D} \varpi_k = \upsilon_j^\ast \eusc{U}^\ast \eusc{D} \eusc{U} \upsilon_k = \langle \xi_j,\xi_k \rangle 
\]
Therefore in order to verify Claim \ref{cla.orthogonal} it suffices to show that $\|\omega_k - \varpi_k\| \leq \delta n$ for all $k \in \{1,\ldots,n\}$. Let $\lambda_1,\ldots,\lambda_n \geq 0$ be the diagonal entries of $\eusc{D}$ and make the following computation. 
\begin{align}\nmbl \sqrt{\eusc{D}}-\eusc{I}_n \nmbr_{\mathrm{HS}}^2 & = \sum_{k=1}^n \Bigl |\sqrt{\lambda_k}-1\Bigr|^2 \label{eq.ortho-0.5} \\
& \leq \sum_{k=1}^n |\lambda_k-1|^2 \label{eq.ortho-1} \\
& = \|\eusc{D}-\eusc{I}_n\|_{\mathrm{HS}}^2  \label{eq.ortho-1.5} \\
& = \|\eusc{U}^\ast(\eusc{D}-\eusc{I}_n)\eusc{U}\|_{\mathrm{HS}}^2 \label{eq.ortho-2} \\
& = \|\eusc{U}^\ast\eusc{D}\eusc{U}-\eusc{I}_n\|_{\mathrm{HS}}^2 \nonumber \\
& = \|\eusc{A}-\eusc{I}_n\|_{\mathrm{HS}}^2  \label{eq.ortho-3} \\
& \leq \delta^2 n^2 \label{eq.ortho-4} 
\end{align}
This computation can be justified as follows. 
\begin{itemize}
\item \eqref{eq.ortho-1} follows from \eqref{eq.ortho-0.5} since $|\sqrt{\lambda}-1| \leq |\lambda-1|$ for all $\lambda \geq 0$. 
\item \eqref{eq.ortho-2} follows from \eqref{eq.ortho-1.5} since the norm $\|\cdot\|_{\mathrm{HS}}$ is unitarily biinvariant. 
\item \eqref{eq.ortho-4} follows from \eqref{eq.ortho-3} since the diagonal entries of both $\eusc{A}$ and $\eusc{I}_n$ are equal to one, while we have assumed the off-diagonal entries of $\eusc{A}$ have magnitude bounded by $\delta$. 
\end{itemize}
For $k \in \{1,\ldots,n\}$ we make the following further computation. 
\begin{align} 
\|\omega_k - \varpi_k\| & = \nmbl \bigl( \sqrt{\eusc{D}} - \eusc{I}_n \bigr) \varpi_k \nmbr \label{eq.ortho-5} \\
& \leq \nmbl \sqrt{\eusc{D}}-\eusc{I}_n \nmbr_{\mathrm{op}} \label{eq.ortho-6} \\
& \leq \nmbl \sqrt{\eusc{D}}-\eusc{I}_n \nmbr_{\mathrm{HS}} \label{eq.ortho-7} \\
& \leq \delta n \label{eq.ortho-8} 
\end{align}
This computation can be justified as follows. 
\begin{itemize} 
\item \eqref{eq.ortho-6} follows from \eqref{eq.ortho-5} since $\varpi_k$ is a unit vector. 
\item \eqref{eq.ortho-7} follows from \eqref{eq.ortho-6} since if $\varsigma_1,\ldots,\varsigma_n \geq 0$ are the singular values of an $n \times n$ matrix $\eusc{B}$ then we have the elementary inequality 
\[
\|\eusc{B}\|_{\mathrm{op}} = \max_{1 \leq k \leq n} \varsigma_k \leq \sqrt{\varsigma_1^2+\cdots+\varsigma_n^2} = \|\eusc{B}\|_{\mathrm{HS}} 
\] 
\item \eqref{eq.ortho-8} follows from \eqref{eq.ortho-7} by the computation leading up to \eqref{eq.ortho-4}. 
\end{itemize}
Taking into account our previous observations, this suffices to establish Claim \ref{cla.orthogonal} and hence the present proposition. \end{proof}

\subsection{Complementary projections and span} \label{subsec.compproj}

The results in Section \ref{subsec.compproj} will allow us to use information about inner products to force linear independence. 
\begin{proposition} \label{prop.s} Let $\xi_1,\ldots,\xi_n$ be a list of unit vectors in a Hilbert space, and suppose that $|\langle \xi_j,\xi_k\rangle| \leq 1/n$ for all distinct pairs $j,k \in \{1,\ldots,n\}$. Then $\xi_1,\ldots,\xi_n$ are linearly independent. \end{proposition}

\begin{proof}[Proof of Proposition \ref{prop.s}] Let $\eusc{G}$ be the Gram matrix of $\xi_1,\ldots,\xi_n$, so we have $\eusc{G}_{j,k} = \langle \xi_j,\xi_k \rangle$ for $j,k \in \{1,\ldots,n\}$. By hypothesis, for $j \in \{1,\ldots,n\}$ we have \[ \sum_{\substack{1 \leq k \leq n \\ k \neq j}} |\eusc{G}_{j,k}| = \sum_{\substack{1 \leq k \leq n \\ k \neq j}} |\langle \xi_j,\xi_k \rangle| \leq \frac{n-1}{n} < 1 = \eusc{G}_{j,j}  \]

Thus $\eusc{G}$ is strictly diagonally dominant and has full rank $n$. Since $\mrm{rank}(\eusc{G}) = \dm(\spn(\xi_1,\ldots,\xi_n))$ the proof of Proposition \ref{prop.s} is complete. \end{proof}

\begin{proposition} \label{prop.compspan} Let $\mathcal{H}$ be a finite dimensional Hilbert space, let $p \in \mrm{Proj}(\mathcal{H})$ be a projection and let $\xi_1,\ldots,\xi_n$ be a list of unit vectors in $\mathcal{H}$. Let $\delta , \kappa > 0$ satisfy $\kappa < 1/10$ and $2\delta \leq 1/n$. Assume that for all $k \in \{1,\ldots,n\}$  we have $\|p\xi_k\| \leq \kappa$ and for all distinct pairs $j,k \in \{1,\ldots,n\}$ we have $|\langle (I_\mathcal{H}-p)\xi_j, \xi_k \rangle| \leq \delta$. Then: 
\begin{description} 
\item[(a)] The vectors $(I_\mathcal{H}-p) \xi_1,\ldots,(I_\mathcal{H}-p) \xi_n$ are linearly independent.
\item[(b)] We have: \[ \spn(p\mathcal{H},\xi_1,\ldots,\xi_n) = p\mathcal{H} \oplus \spn((I_\mathcal{H}-p) \xi_1,\ldots,(I_\mathcal{H}-p) \xi_n)) \]
\item[(c)] If we set 
\[
\varphi_k = \frac{(I_\mathcal{H}-p)\xi_k}{\sqrt{1-\|p\xi_k\|^2}} 
\]
then there exists an orthonormal basis $\vartheta_1,\ldots,\vartheta_n$ for $\spn((I_\mathcal{H}-p) \xi_1,\ldots,(I_\mathcal{H}-p) \xi_n)$ such that 
\[
\max_{1 \leq k \leq n} \|\varphi_k - \vartheta_k\| \leq 2 \delta n
\]   
\end{description} 
\end{proposition}
\begin{proof}[Proof of Proposition \ref{prop.compspan}]
Notice that for  $k \in \{1,\ldots,n\}$  we have: \[ \sqrt{1-||p\xi_j||^2} \geq \sqrt{1-\kappa^2} \geq \frac{\sqrt{99}}{10} \geq \frac{1}{\sqrt{2}} \] Hence for all distinct pairs $j,k \in \{1,\ldots,n\}$ we have: 
\begin{equation} \label{eq.today-3} 
|\langle \varphi_j, \varphi_k \rangle| = \frac{|\langle (I_\mathcal{H}-p)\xi_j, \xi_k \rangle|}{\sqrt{(1-\|p\xi_j\|^2)(1-\|p\xi_k\|^2)}}  \leq 2 |\langle (I_\mathcal{H}-p)\xi_j, \xi_k \rangle| \leq 2\delta \leq \frac{1}{n}  \end{equation} 
Since  
\[
\spn((I_\mathcal{H}-p) \xi_1,\ldots,(I_\mathcal{H}-p) \xi_n) = \spn(\varphi_1,\ldots,\varphi_n) 
\]
we see that Clauses (a) and (b) of Proposition \ref{prop.compspan} follow directly from from Proposition \ref{prop.s}. Clause (c) of Proposition \ref{prop.compspan} follows by combining \eqref{eq.today-3} with Proposition \ref{prop.orthogonal}.  
\end{proof}

\subsection{Almost commuting unitaries and projections}

The result below is a `repair principle' for turning approximately invariant subspaces into exactly invariant subspaces.

\begin{proposition} \label{prop.unitcom} Let $\mathcal{H}$ be a finite dimensional Hilbert space, let $0<\delta<1/2$, let $p\in \mrm{Proj}(\mathcal{H})$ be a projection and let $u\in\mrm{Un}(\mathcal{H})$ be a unitary operator such that $\|(I_{\mathcal{H}}-p)up\|^2_{\mrm{HS}}\leq \delta \dm(\mathcal{H})$.
Then there exists an operator $v\in\mrm{Un}(\mathcal{H})$ commuting with $p$ such that  $v v^{*}=v^{*}v=p$ and $\|(u-v)p\|^2_{\mrm{HS}}\leq 4\delta \dm(\mathcal{H})$.
\end{proposition}
\begin{proof}
Define $w=pup$, so that $w$ commutes with $p$. We have
\[
\|(w-u)p\|^2_{\mrm{HS}}=\|pup-up\|^2_{\mrm{HS}}=\|(I_{\mathcal{H}}-p)up\|^2_{\mrm{HS}}\leq \delta \dm(\mathcal{H}).
\]
Moreover, we have:
\begin{equation}\label{3.8 from the original version}
\|w^{*}w-p\|^2_{\mrm{HS}}=\|pu^{*}pup-p\|^2_{\mrm{HS}} = || pu^\ast(I_\mathcal{H} - p)up ||_{\mrm{HS}}^2 \leq\|(I_{\mathcal{H}}-p)up\|^2_{\mrm{HS}}\leq \delta \dm(\mathcal{H}).
\end{equation}
Now, write $j$ for $\tr(p)$ and choose an orthonormal basis $\vartheta_1,\dots,\vartheta_d$ for $\mathcal{H}$ such that $\vartheta_1,\dots,\vartheta_j$ forms an orthonormal basis for the range of $p$. Since $p$ and $w$ commute, according to the singular value decomposition we can find unitary operators $a$ and $b$ such that $a^{*}pb$ and $a^{*}wb$ are both diagonal in this basis. Let $\varsigma_1,\dots, \varsigma_j$ be the nonzero entries of $a^{*}wb$. Since $p\vartheta_k = \vartheta_k$ for $k \in \{1,\dots, j\}$ we have that the first $j$ diagonal entries of $a^{*}pb$ are equal to one, and so (\ref{3.8 from the original version}) implies:
\[
\left||\varsigma_1|^2-1 \right|+\dots+\left||\varsigma_j|^2-1 \right|\leq \delta \dm(\mathcal{H}).
\]
Using Markov's inequality, we can find a set $S\subset\{1,\dots,j\}$ with $|S| \geq j- \delta\dm(\mathcal{H})$ such that
\[
\left||\varsigma_k|^2-1 \right|\leq \delta
\]
for all $k\in S$. Thus if $k \in S$ we have:
\[
\left|\varsigma_k-\frac{\varsigma_k}{|\varsigma_k|}\right|\leq \frac{|\varsigma_k^2-\varsigma_k|}{|\varsigma_k|}\leq \frac{|\varsigma_k^2-1|}{|\varsigma_k|}\leq \frac{\delta}{1-\delta} \leq 2\delta.
\]
Take $c$ to a $d \times d$ matrix which is diagonal in the basis $\vartheta_1,\ldots,\vartheta_d$, and whose $k^{\mrm{th}}$ diagonal entry is equal to $\varsigma_k /|\varsigma_k|$ when $k \in \{1,\dots, j\}$ and equal to zero if $k \in \{j+1,\ldots,d\}$. Then we find:
\[
\|c-a^{*}wb\|^2_{\mrm{HS}}\leq 2(j-|S|)+\sum_{k\in S} \left|\varsigma_k-\frac{\varsigma_k}{|\varsigma_k|}\right|\leq 4\delta \dm(\mathcal{H})
\] Choosing $v = acb^\ast$ completes the proof of Proposition \ref{prop.unitcom}
\end{proof}

\section{Analysis on high dimensional spheres}  \label{sec.concen}

\subsection{Normalized traces and averages of matrix coefficients} \label{subsec.prob}

If $\mathcal{H}$ is a Hilbert space we let $\mathbb{S}_\mathcal{H}$ denote the unit sphere in $\mathcal{H}$ and if $\mathcal{H}$ has finite dimension we endow $\mathbb{S}_\mathcal{H}$ with the uniform probability measure $\sigma_\mathcal{H}$. 

\begin{remark} \label{rem.rvsc} To preempt certain minor ambiguities involving real versus complex spaces in the statement of Theorem \ref{thm.conc} from Section \ref{subsec.conc} below, we elaborate the construction of $\sigma_\mathcal{H}$ as follows. Firstly, recall our standing convention that $\mathrm{dm}(\mathcal{H})$ denotes the complex dimension of the complex Hilbert space $\mathcal{H}$. Writing $d$ for $\mathrm{dm}(\mathcal{H})$, we have that $\sigma_\mathcal{H}$ is the unique Borel probability measure on $\mathbb{S}_\mathcal{H}$ which is proportional to the pullback of $(2d-1)$-dimensional Hausdorff measure from the unit sphere in $\mathbb{R}^{2d}$ through any choice of identifications $\mathcal{H} \cong \mathbb{C}^d \cong \mathbb{R}^{2d}$. \end{remark}

We have the following basic result about integrals of bilinear forms in random vectors.

\begin{proposition} \label{prop.zero} If $\eta,\vartheta \in \mathcal{H}$ are orthogonal then we have: \[ \int_{\mathbb{S}_\mathcal{H}} \langle \xi, \eta \rangle \ov{\langle \xi, \vartheta \rangle} \deee \sigma_\mathcal{H}(\xi) = 0  \] \end{proposition}

\begin{proof}[Proof of Proposition \ref{prop.zero}] Since $\eta$ and $\vartheta$ are orthogonal, we can find a unitary operator $u \in \mrm{Un}(\mathcal{H})$ such that $u \eta = -\eta$ and $u \vartheta = \vartheta$. Therefore\begin{align*} \int_{\mathbb{S}_\mathcal{H}} \langle \xi, \eta \rangle \ov{\langle \xi, \vartheta \rangle} \deee \sigma_\mathcal{H}(\xi) & = \int_{\mathbb{S}_\mathcal{H}} \langle u^\ast \xi, \eta \rangle \ov{\langle u^\ast \xi, \vartheta \rangle} \deee \sigma_\mathcal{H}(\xi) \\ & = \int_{\mathbb{S}_\mathcal{H}} \langle  \xi, u \eta \rangle \ov{\langle \xi, u \vartheta \rangle} \deee \sigma_\mathcal{H}(\xi)  \\ & =  -\int_{\mathbb{S}_\mathcal{H}} \langle \xi, \eta \rangle \ov{\langle \xi, \vartheta \rangle} \deee \sigma_\mathcal{H}(\xi) \end{align*} where the first equality in the above chain follows from the fact that the measure $\sigma_\mathcal{H}$ is invariant under the unitary operator $u^\ast$. \end{proof}

The next proposition connects this probabilistic structure to traces.

\begin{proposition}\label{prop.tr} 
Let $\mathcal{H}$ be a finite dimensional Hilbert space and let $s \in \mrm{End}(\mathcal{H})$ be a linear operator. Then we have 
\[
\frac{\tr(s)}{\dm(\mathcal{H})} = \int_{\mathbb{S}_\mathcal{H}} \langle s \xi, \xi \rangle \deee \sigma_\mathcal{H}(\xi)  
\]
\end{proposition}

\begin{proof}[Proof of Proposition \ref{prop.tr}] Write $d = \dm(\mathcal{H})$ and let $\vartheta_1,\ldots,\vartheta_d$ be an orthonormal basis for $\mathcal{H}$. According to the Schur triangularization theorem we can find an operator $t$ which is upper-triangular in this basis and unitary operator $u$ such that $u^\ast t u = s$. We have \begin{align}  \int_{\mathbb{S}_\mathcal{H}} \langle s \xi, \xi \rangle  \deee \sigma_\mathcal{H}(\xi) \nonumber & =  \int_{\mathbb{S}_\mathcal{H}}  \langle u^\ast t u \xi,\xi \rangle \deee \sigma_\mathcal{H}(\xi) \nonumber \\ &  =   \int_{\mathbb{S}_\mathcal{H}}  \langle tu \xi,u \xi\rangle \deee \sigma_\mathcal{H}(\xi) \label{eq.tr-0} \\ & =  \int_{\mathbb{S}_\mathcal{H}}  \langle t\xi,\xi \rangle  \deee \sigma_\mathcal{H}(\xi)  \label{eq.tr-1} \end{align} where \eqref{eq.tr-1} follows from \eqref{eq.tr-0} since the measure $\sigma_\mathcal{H}$ is $u$-invariant. Write $\lambda_1,\ldots,\lambda_d$ for the diagonal entries of $t$, which are the eigenvalues of $s$. We calculate  \begin{align} \int_{\mathbb{S}_\mathcal{H}}  \langle t\xi,\xi \rangle  \deee \sigma_\mathcal{H}(\xi) & = \int_{\mathbb{S}_\mathcal{H}}  \left  \langle  t \sum_{k=1}^d \langle \xi, \vartheta_k \rangle  \vartheta_k , \sum_{\ell=1}^d \langle \xi, \vartheta_\ell \rangle \vartheta_\ell  \right \rangle  \deee \sigma_\mathcal{H}(\xi) \label{eq.tr-2} \\ & = \int_{\mathbb{S}_\mathcal{H}}  \sum_{k,\ell =1}^d    \langle t \vartheta_k,\vartheta_\ell \rangle  \langle \xi, \vartheta_k \rangle \ov{ \langle \xi, \vartheta_\ell \rangle}  \deee \sigma_\mathcal{H}(\xi) \label{eq.tr-3} \\ & = \sum_{k,\ell =1}^d \langle t \vartheta_k,\vartheta_\ell \rangle   \int_{\mathbb{S}_\mathcal{H}}     \langle \xi, \vartheta_k \rangle \ov{ \langle \xi, \vartheta_\ell \rangle}  \deee \sigma_\mathcal{H}(\xi) \label{eq.tr-3.5} \\ & =  \sum_{k =1}^d  \langle t \vartheta_k,\vartheta_k \rangle \int_{\mathbb{S}_\mathcal{H}}     |\langle \xi, \vartheta_k \rangle|^2  \deee \sigma_\mathcal{H}(\xi) \label{eq.tr-3.75} \\ & = \sum_{k =1}^d    \lambda_k \int_{\mathbb{S}_\mathcal{H}}  |\langle \xi, \vartheta_k \rangle|^2  \deee \sigma_\mathcal{H}(\xi)  \label{eq.tr-5} \end{align} Here, \begin{itemize} \item the equality in \eqref{eq.tr-2} follows from applying the Plancherel theorem to  the random vector $\xi$, \item \eqref{eq.tr-3.75} follows from \eqref{eq.tr-3.5} by Proposition \ref{prop.zero} since $\vartheta_k$ and $\vartheta_\ell$ are orthogonal for all distinct pairs $k,\ell \in \{1,\ldots,d\}$, \item  and \eqref{eq.tr-5} follows from \eqref{eq.tr-3.75} since $t$ is upper-triangular for $\vartheta_1,\ldots,\vartheta_d$ with diagonal entries $\lambda_1,\ldots,\lambda_d$.  
\end{itemize}
Since the measure $\sigma_\mathcal{H}$ is uniform on $\mathbb{S}_\mathcal{H}$, we have 
\begin{equation} 
\int_{\mathbb{S}_\mathcal{H}}  |\langle \xi,\eta\rangle|^2 \deee \sigma_\mathcal{H}(\xi)   = \int_{\mathbb{S}_\mathcal{H}} |\langle \xi,\omega\rangle|^2  \deee \sigma_\mathcal{H}(\xi) \label{eq.tr-9}  
\end{equation}
for all $\eta,\omega \in \mathbb{S}_\mathcal{H}$. On the other hand, Parseval's identity gives 
\begin{equation}
\int_{\mathbb{S}_\mathcal{H}} \sum_{k=1}^d |\langle \xi,\vartheta_k\rangle|^2 \deee \sigma_\mathcal{H}(\xi)  = \int_{\mathbb{S}_\mathcal{H}} \|\xi\|^2  \deee \sigma_\mathcal{H}(\xi)  = 1 \label{eq.tr-10} 
\end{equation}
From \eqref{eq.tr-9} and \eqref{eq.tr-10} we find  
\[
\int_{\mathbb{S}_\mathcal{H}} |\langle \xi,\vartheta_k\rangle|^2 \deee \sigma_\mathcal{H}(\xi)  = \frac{1}{d} 
\]
for all $k \in \{1,\ldots,d\}$. Combining the previous display with \eqref{eq.tr-5} we obtain 
\[
\int_{\mathbb{S}_\mathcal{H}}  \langle t\xi,\xi \rangle \deee \sigma_\mathcal{H}(\xi)  = \frac{1}{d} \sum_{k=1}^d \lambda_k 
\]
and then from \eqref{eq.tr-1} we find 
\[
\int_{\mathbb{S}_\mathcal{H}} \langle s \xi, \xi \rangle   \deee \sigma_\mathcal{H}(\xi)  = \frac{1}{d} \sum_{k=1}^d \lambda_k. 
\]
Since $\lambda_1,\ldots,\lambda_d$ are the eigenvalues of $s$, we obtain 
\[
\int_{\mathbb{S}_\mathcal{H}}  \langle s \xi, \xi \rangle   \deee \sigma_\mathcal{H}(\xi) = \frac{\tr(s)}{d} 
\]
as required. \end{proof}

For $s \in \mrm{End}(\mathcal{H})$ and $\xi \in \mathcal{H}$ we have $\|s \xi\|^2 = \langle s^\ast s \xi, \xi \rangle$ and by definition we have $\tr(s^\ast s) = \|s\|_{\mrm{HS}}^2$. Thus, we obtain the following corollary.

\begin{corollary}\label{cor.tr-norm} Let $\mathcal{H}$ be a finite dimensional Hilbert space and let $s \in \mrm{End}(\mathcal{H})$. Then we have \[ \frac{\|s\|_{\mrm{HS}}^2}{\dm(\mathcal{H})} = \int_{\mathbb{S}_\mathcal{H}} \|s\xi\|^2 \deee \sigma_\mathcal{H}(\xi)  \] \end{corollary}

\subsection{Concentration of measure and deviation bounds} \label{subsec.conc}

For a Hilbert space $\mathcal{H}$, we let $D_\mathcal{H}$ denote the geodesic distance on $\mathbb{S}_\mathcal{H}$ given by $D_\mathcal{H}(\xi,\eta) = \arccos(\langle \xi, \eta\rangle)$ for $\xi,\eta \in \mathbb{S}_\mathcal{H}$. Given a function $f:\mathbb{S}_\mathcal{H} \to \mathbb{C}$, we introduce the $D_\mathcal{H}$-Lipschitz constant of $f$ defined by \[ \|f\|_{\mathrm{Lip}} = \sup_{\xi,\eta \in \mathbb{S}_\mathcal{H}: \xi \neq \eta} \frac{|f(\xi)-f(\eta)|}{D_\mathcal{H}(\xi,\eta)} \] and we refer to $f$ as $D_\mathcal{H}$-Lipschitz when this supremum is finite. The so-called concentration of measure phenomenon asserts that if $\mathcal{H}$ has large finite dimension then the measure $\sigma_\mathcal{H}$ places most of its mass on small $D_\mathcal{H}$-neighborhoods of higher-dimensional meridians.  A general reference on the topic is \cite{MR1849347}. \\
\\
One of many useful consequences of this phenomenon is that it provides a way to control the deviation of $D_\mathcal{H}$-Lipschitz functions on $\mathbb{S}_\mathcal{H}$ from their averages with respect to $\sigma_\mathcal{H}$. Such a bound appropriate for our purposes is stated below as Theorem \ref{thm.conc} and appears with proof as Theorem 1.7.9 in \cite{MR3185453}. (Here we take into account Remark \ref{rem.rvsc}.)

\begin{theorem}[Concentration of measure in spheres] \label{thm.conc} For all finite dimensional Hilbert spaces $\mathcal{H}$, all $c > 0$ and all $D_\mathcal{H}$-Lipschitz functions $f:\mathbb{S}_\mathcal{H} \to \mathbb{C}$ we have \begin{equation} \sigma_\mathcal{H}\left (\left \{ \xi \in \mathbb{S}_\mathcal{H}: f(\xi) \approx_c \int_{\mathbb{S}_\mathcal{H}} f(\eta) \deee \sigma_\mathcal{H}(\eta) \right \} \right ) \geq 1- 2\exp \left(-\frac{c^2(2 \dm(\mathcal{H})-1)}{2\,\|f\|_{\mathrm{Lip}}^2 } \right )\label{eq.conc} \end{equation} \end{theorem}

Here and in the sequel, for $w,z \in \mathbb{C}$ and $c>0$ we use the notation $w \approx_c z$ to mean $|w-z|\leq c$.

\begin{remark} We will not require the explicit exponential concentration provided by Theorem \ref{thm.conc}, and for our applications it would suffice to know that the right side of \eqref{eq.conc} converges to one uniformly in $c^{-1}$ and $\|f\|_{\mathrm{Lip}}$ as $\dm(\mathcal{H})$ approaches infinity. \end{remark} Theorem \ref{thm.conc} will combine with Proposition \ref{prop.tr} to control the deviation of a matrix coefficient of a normal operator in a high dimensional space from its normalized trace as follows.

\begin{proposition} \label{prop.dev-1} For all $c > 0$, all finite dimensional Hilbert spaces $\mathcal{H}$ and all linear operators $s \in \mrm{End}(\mathcal{H})$ we have  
\[
\sigma_\mathcal{H} \left ( \left \{ \xi \in \mathbb{S}_\mathcal{H}: \langle s \xi, \xi \rangle \approx_c \frac{\tr(s)}{\dm(\mathcal{H})} \right \} \right ) \geq 1- 2\exp \left(-\frac{c^2(2 \dm(\mathcal{H})-1)}{8\,\|s\|_{\mrm{op}}^2 } \right ) 
\]
\end{proposition}

\begin{proof}[Proof of Proposition \ref{prop.dev-1}]  Proposition \ref{prop.tr} implies 
\begin{equation} \label{eq.devv-1} 
\int_{\mathbb{S}_\mathcal{H}} \langle s \eta, \eta \rangle  \deee \sigma_\mathcal{H}(\eta) = \frac{\tr(s)}{\dm(\mathcal{H})} \end{equation} Now, let $\xi,\eta \in \mathbb{S}_\mathcal{H}$ and calculate \begin{align}|\langle  s \xi, \xi \rangle - \langle s \eta, \eta \rangle| & \leq |\langle s \xi, \xi \rangle  - \langle s \eta, \xi \rangle|+ |\langle s \eta, \xi \rangle  - \langle s \eta, \eta \rangle| \nonumber  \\ & = |\langle s(\xi-\eta),\xi\rangle| + |\langle s \eta, \xi-\eta \rangle| \nonumber \\ & \leq \|s(\xi-\eta)\| + \|s \eta\| \, \|\xi-\eta\| \nonumber \\ & \leq 2 \|s\|_{\mrm{op}} \| \xi - \eta\| \nonumber \\ & \leq 2 \|s\|_{\mrm{op}} D_\mathcal{H}(\xi,\eta) \label{eq.devv-2} \end{align} where the last inequality holds since \[ D_\mathcal{H}(\xi,\eta) = 2 \arcsin\left(\frac{\|\xi-\eta\|}{2}\right) \geq \|\xi-\eta\| \] The inequalities leading up to \eqref{eq.devv-2} can be interpreted as an assertion that the map $\xi \mapsto \langle s \xi, \xi \rangle$ is a $D_\mathcal{H}$-Lipschitz function on $\mathbb{S}_\mathcal{H}$ with Lipschitz constant $2\,\|s\|_{\mrm{op}}$, so Proposition \ref{prop.dev-1} now follows by combining \eqref{eq.devv-1} with Theorem \ref{thm.conc}. \end{proof}

By substituting $s^\ast s$ for $s$ in Proposition \ref{prop.dev-1} and using Corollary \ref{cor.tr-norm} we obtain the following.

\begin{corollary} \label{cor.dev-1} For all $c > 0$, all finite dimensional Hilbert spaces $\mathcal{H}$ and all linear operators $s \in \mrm{End}(\mathcal{H})$ we have  
\[
\sigma_\mathcal{H} \left ( \left \{ \xi \in \mathbb{S}_\mathcal{H}: \|s \xi\|^2 \leq c+ \frac{\|s\|_{\mrm{HS}}^2}{\dm(\mathcal{H})} \right \} \right ) \geq 1- 2\exp \left(-\frac{c^2(2 \dm(\mathcal{H})-1)}{8\,\|s\|_{\mrm{op}}^4 } \right ) 
\]
\end{corollary}

\section{Disuniformity of measures} \label{sec.disunif}

\subsection{Definition and integral estimate}

We will typically identify the interval $[0,1)$ with the unit circle $\mathbb{T}$ via the map $x \mapsto e^{2\pi ix}$. We also write $\mrm{Prob}(\mathbb{T})$ for the set of Borel probability measures on $\mathbb{T}$. We consider the total variation distance between elements $\mu,\nu$ of $\mrm{Prob}(\mathbb{T})$ defined by: 
\[
\|\mu-\nu\|_{\mrm{TV}} = \sup \Bigl \{ |\mu(E) - \nu(E)|: E \mbox{ is a Borel subset of }\mathbb{T} \Bigr\} 
\]

For $N \in \mathbb{N}$ and $k \in \{0,\ldots,N-1\}$ it will be convenient to adopt the notation $\mathscr{I}_N[k]$ for the subinterval $[k/N,(k+1)/N)$ of $\mathbb{T}$, and to consider the function $\mathscr{P}_N:\mathbb{T} \to \{0,\ldots,N-1 \}$ defined by letting $\mathscr{P}_N(x)$ be the unique element of $\{0,\ldots,N-1 \}$ such that $x \in \mathscr{I}_N[\mathscr{P}_N(x)]$

\begin{definition}  \label{def.disun}
Let $\mu \in \mrm{Prob}(\mathbb{T})$ and let $N \in \mathbb{N}$. We define the \textbf{resolution-}$N$\textbf{ disunifomity of }$\mu$ by: 
\[
\dss_N(\mu) = \sum_{k=0}^{N-1} \left \vert \mu(\mathscr{I}_N[k]) - \frac{1}{N}\right \vert 
\] 
\end{definition}

The above definition can control the integrals of relatively smooth functions as follows.

\begin{proposition} \label{prop.integral} For any $\mu \in \mrm{Prob}(\mathbb{T})$ and any $\psi \in C^1(\mathbb{T})$ we have \[ \left \vert  \int_\mathbb{T} \psi(x) \deee \mu(x) - \int_\mathbb{T} \psi(x) \deee x \right \vert \leq  \dss_N(\mu) \, \sup_{x \in \mathbb{T}} |\psi(x)| +\frac{2}{N}   \sup_{x \in \mathbb{T}} |\psi'(x)|   \] \end{proposition}

\begin{proof}[Proof of Proposition \ref{prop.integral}] Define a `discretized' version $\psi_\circ$ of $\psi$ as follows: \[ \psi_\circ(x) = \psi\left( \frac{\mathscr{P}_N(x)}{N}\right)  \]

According to the mean value theorem, for all $x \in \mathbb{T}$ we have: \[ |\psi(x) - \psi_\circ(x)| \leq \frac{1}{N} \sup_{x \in \mathbb{T}} |\psi'(x)|   \] and thus for any $\nu \in \mrm{Prob}(\mathbb{T})$ we have: \begin{equation} \label{eq.int-1} \left \vert \int_\mathbb{T} \psi(x)  \deee \nu(x)  - \int_\mathbb{T} \psi_\circ(x)  \deee \nu(x) \right \vert  \leq  \frac{1}{N} \sup_{x \in \mathbb{T}} |\psi'(x)|  \end{equation}

Moreover, since $\psi_\circ$ is constant on the intervals $\mathscr{I}_N[\cdot]$, we have: \[ \int_\mathbb{T} \psi_\circ \deee \mu(x) = \sum_{k=0}^{N-1} \psi\left( \frac{k}{N}\right) \mu(\mathscr{I}_N[k])  \] Thus we find: \begin{equation} \left \vert \int_\mathbb{T} \psi_\circ(x)  \deee \mu(x)  - \int_\mathbb{T} \psi_\circ(x)  \deee x \right \vert  =   \left \vert \sum_{k=0}^{N-1} \psi\left( \frac{k}{N}\right) \left( \mu(\mathscr{I}_N[k])  - \frac{1}{N} \right)  \right \vert \leq \dss_N(\mu) \, \sup_{x \in \mathbb{T}} |\psi(x)| \label{eq.int-2} \end{equation} We calculate: \begin{align} \label{eq.int-3}  \left \vert  \int_\mathbb{T} \psi(x) \deee \mu(x) - \int_\mathbb{T} \psi(x) \deee x \right \vert  & \leq  \left \vert \int_\mathbb{T} \psi_\circ(x)  \deee \nu(x)  - \int_\mathbb{T} \psi_\circ(x)  \deee x \right \vert  +\frac{2}{N}   \sup_{x \in \mathbb{T}} |\psi'(x)| \\ & \leq  \dss_N(\mu) \, \sup_{x \in \mathbb{T}} |\psi(x)| +\frac{2}{N}   \sup_{x \in \mathbb{T}} |\psi'(x)| \label{eq.int-4}  \end{align} Here, the inequality in \eqref{eq.int-3} follows from by combining the triangle inequality with the instances of \eqref{eq.int-1} for both $\mu$ and Lebesgue measure, while \eqref{eq.int-4} follows from \eqref{eq.int-3} by \eqref{eq.int-2}.
	
\end{proof}

\subsection{Total variation estimate}

We will need an alternative characterization of total variation distance (see Section 2 of \cite{419a6ad1-1c79-3c86-94d2-f62a47a8c518}): 
\begin{equation} \label{eq.teevee-alt} |\mu-\nu\|_{\mrm{TV}} = \frac{1}{2}\sup \left \{ \left \vert \int_{\mathbb{T}} \psi(x) \deee \mu(x) - \int_{\mathbb{T}} \psi(x) \deee \nu(x) \right \vert : \psi \mbox{ is a Borel function from }\mathbb{T}\mbox{ to }[-1,1] \right \} 
\end{equation}

The next result implements the heuristic that if two measures are close in total variation and one is roughly uniform, then the other must also be roughly uniform.

\begin{proposition} \label{prop.teevee} For any $\mu,\nu \in \mrm{Prob}(\mathbb{T})$ we have: \[ |\dss_N(\mu) - \dss_N(\nu)| \leq 2\|\mu-\nu\|_{\mrm{TV}} \] \end{proposition}

\begin{proof}[Proof of Proposition \ref{prop.teevee}]
Define a function $\psi:\mathbb{T} \to \{-1,1\}$ by setting 
\[ 
\psi(x) = 
\begin{cases} -1 & \mbox{ if }\nu(\mathscr{I}_N[\mathscr{P}_N(x)]) > \mu(\mathscr{I}_N[\mathscr{P}_N(x)]) \\ \,\,\, 1 & \mbox{ if }\nu(\mathscr{I}_N[\mathscr{P}_N(x)]) \leq \mu(\mathscr{I}_N[\mathscr{P}_N(x)])  
\end{cases} 
\] 
This $\psi$ is constructed such that the following holds for all $k \in \{0,\ldots,N-1\}$.  
\begin{equation} \label{eq.eff} 
|\mu(\mathscr{I}_N[k]) - \nu(\mathscr{I}_N[k])| = \int_{\mathscr{I}_N[k]} \psi(x) \deee \mu(x) - \int_{\mathscr{I}_n[k]} \psi(x) \deee \nu(x)  
\end{equation} 
Thus we have 
\begin{equation} \label{eq.teevee-3} 
\sum_{k=0}^{N-1} |\mu(\mathscr{I}_N[k]) - \nu(\mathscr{I}_N[k])| = \int_\mathbb{T} \psi(x) \deee \mu(x) - \int_\mathbb{T} \psi(x) \deee \nu(x)  \leq 2\|\mu-\nu\|_{\mrm{TV}} 
\end{equation} 
where the equality on the left follows from \eqref{eq.eff} and the inequality on the right follows from \eqref{eq.teevee-alt}. We compute: 
\begin{align} 
\label{eq.gee-1} |\dss_N(\mu) - \dss_N(\nu)| & = \left \vert   \sum_{k=0}^{N-1} \left( \left \vert \mu(\mathscr{I}_N[k]) - \frac{1}{N} \right \vert -  \left \vert \nu(\mathscr{I}_N[k]) - \frac{1}{N} \right \vert\right)  \right \vert \\ & \label{eq.gee-2}  \leq  \sum_{k=0}^{N-1} \left \vert  \left \vert \mu(\mathscr{I}_N[k]) - \frac{1}{N} \right \vert -  \left \vert \nu(\mathscr{I}_N[k]) - \frac{1}{N} \right \vert  \right \vert \\ & \label{eq.gee-3}  \leq  \sum_{k=0}^{N-1} |\mu(\mathscr{I}_N[k]) - \nu(\mathscr{I}_N[k])| \\ & \label{eq.gee-4}  \leq 2\|\mu - \nu\|_{\mrm{TV}}  
\end{align} 
Here, 
\begin{itemize} 
\item the equality in \eqref{eq.gee-1} follows from the definition of $\dss_N$, 
\item \eqref{eq.gee-3} follows from \eqref{eq.gee-2} by the reverse triangle inequality, 
\item and \eqref{eq.gee-4} follows from \eqref{eq.gee-3} by \eqref{eq.teevee-3}. 
\end{itemize} This completes the proof of Proposition \ref{prop.teevee}. 
\end{proof}

\subsection{Convexity estimates}

The following is immediate from the definition and the triangle inequality.

\begin{proposition} \label{prop.convex} Let $\mu_1,\ldots,\mu_n \in \mrm{Prob}(\mathbb{T})$ and let $c_1,\ldots,c_n \in [0,1]$ satisfy $c_1+\cdots+c_n = 1$. Then we have: \[ \dss_N(c_1\mu_1+\cdots+c_n \mu_n) \leq  c_1 \dss_N(\mu_1)+\cdots+c_n \dss_N(\mu_n) \] \end{proposition}

Given measure $\mu$ and $\nu$ on $\mathbb{T}$, we use the notation $\nu \leq \mu$ to mean that $\nu(E) \leq \mu(E)$ for every Borel subset $E$ of $\mathbb{T}$.

\begin{proposition} \label{prop.subtract} Suppose $\mu,\nu \in \mrm{Prob}(\mathbb{T})$ and $c \in (0,1)$ is such that $c \nu \leq \mu$. Then we have: \[ \dss_n\left(\frac{\mu-c\nu}{1-c}\right) \leq \frac{\dss_N(\mu)+ c \dss_N(\nu)}{1-c}  \] \end{proposition}

\begin{proof}[Proof of Proposition \ref{prop.subtract}]
We compute 
\begin{align*} 
\dss_n\left(\frac{\mu-c\nu}{1-c}\right) & =  \sum_{k=0}^{N-1} \left \vert \frac{\mu(\mathscr{I}_N[k]) - c\nu(\mathscr{I}_N[k])}{1-c} -\frac{1}{N}  \right \vert \\
& =  \sum_{k=0}^{N-1}  \left \vert \frac{\mu(\mathscr{I}_N[k])}{1-c} -\frac{1}{(1-c)N} - \frac{c\nu(\mathscr{I}_N[k])}{1-c}  + \frac{c}{(1-c)N}  \right \vert \\
& \leq \frac{1}{1-c} \left( \sum_{k=0}^{N-1} \left \vert \mu(\mathscr{I}_N[k]) - \frac{1}{N} \right \vert  \right) + \frac{c}{1-c} \left( \sum_{k=0}^{N-1} \left \vert \nu(\mathscr{I}_N[k]) - \frac{1}{N} \right \vert  \right) \\
& = \frac{\dss_N(\mu)+ c \dss_N(\nu)}{1-c}  
\end{align*} 
	
\end{proof}
 
\section{Localized spectral measures} \label{sec.locspec}

\subsection{Definition and relation with traces} \label{sec.loccc}

For $x \in \mathbb{T}$ we write $\boldsymbol{\delta}_x$ for the element of $\mrm{Prob}(\mathbb{T})$ corresponding to the pure point measure at $x$.

\begin{definition}
Let $\mathcal{H}$ be a Hilbert space with finite dimension $d \in \mathbb{N}$, let $\mathcal{V}$ be a nonzero subspace of $\mathcal{H}$ and let $p \in \mrm{Proj}(\mathcal{H})$ be the projection onto $\mathcal{V}$. Also let $u \in \mrm{Un}(\mathcal{H})$ be a unitary operator with eigenvalues $\lambda_1,\ldots,\lambda_d \in \mathbb{T}$ and associated orthonormal eigenvectors $\vartheta_1,\ldots,\vartheta_d$. We define the $\mathcal{V}$\textbf{-localized spectral measure of }$u$ to be the element of $\mrm{Prob}(\mathbb{T})$ given by: 
\[
\mathbf{S}_u(\mathcal{V}) = \frac{\|p\vartheta_1\|^2 \boldsymbol{\delta}_{\lambda_1} + \cdots + \|p\vartheta_d\|^2 \boldsymbol{\delta}_{\lambda_d}}{\dm(\mathcal{V})} 
\]
We may also write $\mathbf{S}_u(p)$ instead of $\mathbf{S}_u(\mathcal{V})$. If $\xi_1,\ldots,\xi_n$ is a list of vectors in $\mathcal{H}$, we write $\mathbf{S}_u(\xi_1,\ldots,\xi_n)$ for $\mathbf{S}_u(\spn(\xi_1,\ldots,\xi_n))$. 
\end{definition}

Propositions \ref{prop.trace} - \ref{prop.ds-trace} below collect some basic properties of these localized spectral measures.

\begin{proposition} \label{prop.trace} Let $\mathcal{H}$ be a finite-dimensional Hilbert space, let $u \in \mrm{Un}(\mathcal{H})$ be a unitary operator and let $p \in \mrm{Proj}(\mathcal{H})$ be a nonzero projection. Then we have: \[ \frac{|\tr(pu)|}{\tr(p)} \leq  \dss_N(\mathbf{S}_u(p)) + \frac{2}{N} \] \end{proposition}

\begin{proof}[Proof of Proposition \ref{prop.trace}] Let $\lambda_1,\ldots,\lambda_d \in \mathbb{T}$ be the eigenvalues of $u$ with corresponding eigenvectors $\vartheta_1,\ldots,\vartheta_d$.  We have 
\[ \tr(pu)  = \sum_{k=1}^d \langle u \vartheta_k, p \vartheta_k \rangle  = \sum_{k=1}^d \lambda_k \|p \vartheta_k\|^2  \]
and therefore \[ \frac{\tr(pu)}{\tr(p)} = \int_\mathbb{T} e^{2\pi i x} \deee[\mathbf{S}_u(p)](x) \] so Proposition \ref{prop.trace} follows from Proposition \ref{prop.integral}. \end{proof}

\begin{proposition} \label{prop.interval} Let $E$ be a Borel subset of $\mathbb{T}$ and let $q$ denote the projection onto the span of the eigenvectors of $u$ with corresponding eigenvalues in $E$. Then we have $\|qp\|_{\mrm{HS}}^2/\tr(p) = [\mathbf{S}_u(p)](E)$. \end{proposition}

\begin{proof}[Proof of Proposition \ref{prop.interval}] Let $\lambda_1,\ldots,\lambda_d \in \mathbb{T}$ be the eigenvalues of $u$ with associated eigenvectors $\vartheta_1,\ldots,\vartheta_d$. We compute: \[ \frac{\|pq\|_{\mrm{HS}}^2}{\tr(p)} = \frac{1}{\tr(p)} \sum_{k=1}^d \|pq \vartheta_k\|^2 = \frac{1}{\tr(p)} \sum_{\substack{1 \leq k \leq d:\\ \lambda_k \in E }} \|p \vartheta_k\|^2 =  [\mathbf{S}_u(p)](E) \] where the right equality follows from the definition of $\mathbf{S}_u(p)$. Since both $p$ and $q$ are projections, we have: \[ \|pq\|_{\mrm{HS}}^2 = \tr((pq)^\ast (pq)) = \tr(qpq) = \tr(pq) = \tr(qp) = \tr(pqp) = \tr((qp)^\ast qp) = \|qp\|_{\mrm{HS}}^2 \]  and thus the proof of Proposition \ref{prop.interval} is complete.  \end{proof}

\begin{proposition} \label{prop.qual-1} Let $p \in \mrm{Proj}(\mathcal{H})$ be a projection, let $\xi \in \mathcal{H}$ be a nonzero vector and let $E$ be a Borel subset of $\mathbb{T}$. Letting $q$ denote the projection onto the span of the eigenvectors of $u$ with eigenvalues in $E$ we have: \[ [\mathbf{S}_u(p\xi)](E) = \frac{\|qp \xi\|^2}{\|p\xi\|^2} \] \end{proposition}

\begin{proof}[Proof of Proposition \ref{prop.qual-1}] 
Write $r$ for the projection onto $\spn(p\xi)$, so for any $\eta \in \mathcal{H}$ we have: 
\begin{equation} \label{eq.arr} 
r\eta = \frac{\langle p \xi, \eta \rangle}{\|p\xi\|^2} p \xi  
\end{equation}
Let $\lambda_1,\ldots,\lambda_d \in \mathbb{T}$ be the eigenvalues of $u$ with associated eigenvectors $\vartheta_1,\ldots,\vartheta_d$ and let 
\[
J = \bigl\{ \ell \in \{1,\ldots,d\}: \lambda_\ell \in E \bigr\} 
\]
We find 
\begin{equation} \label{eq.qual-1}
[\mathbf{S}_u(p\xi)](E)  = \sum_{\ell \in J} \|r \vartheta_k\|^2 = \frac{1}{\|p\xi\|^2} \sum_{\ell \in J} |\langle p \xi, \vartheta_\ell \rangle|^2 
\end{equation}
where the left equality above follows from the definition of $\mathbf{S}_u(p\xi)$ and the right equality follows from \eqref{eq.arr}. Now, since $\{\vartheta_\ell: \ell \in J\}$ is an orthonormal basis for the range of $q$, we have
\[
qp\xi = \sum_{\ell \in J} \langle p\xi, \vartheta_\ell \rangle \vartheta_\ell 
\]
and so Parseval's identity gives: 
\[
\|qp\xi\|^2 = \sum_{\ell \in J} |\langle p \xi ,\vartheta_\ell \rangle|^2 
\] 
Combining the previous display with \eqref{eq.qual-1} we find 
\begin{equation} 
[\mathbf{S}_u(p\xi)](E) = \frac{\|qp \xi\|^2}{\|p\xi\|^2} 
\end{equation} as required. 
\end{proof}

\begin{proposition} \label{prop.ds-trace} 
For all $\epsilon > 0$ and all $N \in \mathbb{N}$ there exists $\delta > 0$ and $M \in \mathbb{N}$ such that for any finite-dimensional Hilbert space $\mathcal{H}$ and any unitary operator $u \in \mrm{Un}(\mathcal{H})$ such that $|\tr(u^k)| \leq \delta \dm(\mathcal{H})$ for all $k \in \{-M,\ldots,M\}$ we have $\dss_N(\mathbf{S}_u(\mathcal{H})) \leq \epsilon$. 
\end{proposition}

\begin{proof}[Proof of Proposition \ref{prop.ds-trace}] Suppose toward a contradiction that Proposition \ref{prop.ds-trace} fails. Then there exist $N\in \mathbb{N}$ and $\epsilon>0$ such that for all  $n \in \mathbb{N}$ we can find a finite-dimensional Hilbert space $\mathcal{H}_n$ and a unitary operator $u_n \in \mrm{Un}(\mathcal{H}_n)$ with the following properties. \begin{itemize} \item We have $|\tr(u^k)| \leq \dm(\mathcal{H})/n$ for all $k \in \{-n,\ldots,n\}$.  \item If we write $\mu_n$ for the element $\mathbf{S}_{u_n}(\mathcal{H}_n)$ of $\mrm{Prob}(\mathbb{T})$ then $\mtt{DS}_N(\mu_n) > \epsilon$. \end{itemize}

Since the space $\mrm{Prob}(\mathbb{T})$ is compact with respect to the vague topology, by passing to a subsequence we may assume that $(\mu_n)_{n \in \mathbb{N}}$ converges in the vague topology to a measure $\mu \in \mrm{Prob}(\mathbb{T})$. This convergence implies that $\hat{\mu}(k) = \lim_{n \to \infty} \hat{\mu_n}(k)$ for all $k \in \mathbb{Z}$. We also have \[ \frac{\tr(u_n^k)}{\dm(\mathcal{H}_n)} = \hat{\mu_n}(k) \] and therefore $|\hat{\mu_n}(k)| \leq 1/n$ for all nonzero $k \in \{-n,\ldots,n\}$. It follows that $\hat{\mu}(k) = 0$ for all nonzero $k \in \mathbb{Z}$. Therefore $\mu$ must coincide with Lebesgue measure and so $\mtt{DS}_N(\mu) = 0$. However, we also have $\mtt{DS}_N(\mu) = \lim_{n \to \infty} \mtt{DS}_N(\mu_n) \geq \epsilon$. This contradiction completes the proof of Proposition \ref{prop.ds-trace}. \end{proof}

\begin{remark} \label{rem.torsion-1} Proposition \ref{prop.ds-trace} is essentially the only place in our argument where modifications to deal with torsion elements are required. More explicitly, assume that the group $G$ for which we aim to prove Theorem \ref{thm.sofic} contains an element $g$ with $g^n = e$ for some $n \in \mathbb{N}$. Then in a hyperlinear approximation $\alpha:G \to \mrm{Un}(\mathcal{H})$ we will find that $\alpha(g)^n$ is close to $I_\mathcal{H}$ and so $\tr(\alpha(g)^n)/ \dm(\mathcal{H})$ will be close to one. This prevents direct application of Proposition \ref{prop.ds-trace}. \\
\\
The appropriate modification is to change Definition \ref{def.disun} to quantify the distance between $\mu$ and a measure which can be either Lebesgue or uniform on the $n^{\mrm{th}}$ roots of unity, depending on order of the group element whose image under a hyperlinear approximation is the target of Proposition \ref{prop.ds-trace}. With this change, all arguments in Sections \ref{sec.disunif} and \ref{sec.locspec} proceed identically. We will again address the issue of torsion elements when it arises in Remark \ref{rem.torsion-2} from Section \ref{sec.projconc} below. \end{remark}

\subsection{Convexity estimates}

The next proposition follows immediately from the definition.

\begin{proposition} \label{prop.subadd} Suppose $\mathcal{V}_1,\ldots,\mathcal{V}_n$ is a list of nonzero subspaces such that $\mathcal{V}_j \perp \mathcal{V}_k$ for all distinct pairs $j,k \in \{1,\ldots,n\}$. Then we have: 
\[
\mathbf{S}_u(\mathcal{V}_1+\cdots+\mathcal{V}_n) = \frac{\dm(\mathcal{V}_1)\mathbf{S}_u(\mathcal{V}_1) + \cdots + \dm(\mathcal{V}_n)\mathbf{S}_u(\mathcal{V}_n)}{\dm(\mathcal{V}_1)+\cdots+\dm(\mathcal{V}_n)} 
\]
\end{proposition}

\begin{proposition} \label{prop.subtract2} Let $\mathcal{V}$ be a subspace of $\mathcal{H}$ and let $\mathcal{W}$ be a proper nonzero subspace of $\mathcal{V}$. With the notation $c = \dm(\mathcal{W})/\dm(\mathcal{V}) \in (0,1)$ we have 
\[
\dss_N(\mathbf{S}_u(\mathcal{V} \ominus \mathcal{W})) \leq \frac{\dss_N(\mathbf{S}_u(\mathcal{V})) + c\dss_N(\mathbf{S}_u(\mathcal{W})) }{1-c}  
\] 
\end{proposition}

\begin{proof}[Proof of Proposition \ref{prop.subtract2}] According to Proposition \ref{prop.subadd} we have 
\[
\mathbf{S}_u(\mathcal{V}) = \frac{\dm(\mathcal{W})\mathbf{S}_u(\mathcal{W}) + \dm(\mathcal{V} \ominus \mathcal{W}) \mathbf{S}_u(\mathcal{V} \ominus \mathcal{W})}{\dm(\mathcal{V})}  
\]
or equivalently: 
\[ \mathbf{S}_u(\mathcal{V}) = c\mathbf{S}_u(\mathcal{W}) + (1-c)\mathbf{S}_u(\mathcal{V} \ominus \mathcal{W}) 
\]
Thus Proposition \ref{prop.subtract2} follows from Proposition \ref{prop.subtract}. 
\end{proof}

\subsection{Perturbation estimates}

Propositions \ref{prop.teeveee} and \ref{prop.street} below control how disuniformity behaves under various perturbations.

\begin{proposition} \label{prop.teeveee} 
Let $\mathcal{H}$ be a finite dimensional Hilbert space, let $u \in \mrm{Un}(\mathcal{H})$ and suppose $\xi,\eta \in \mathcal{H}$ are unit vectors. Then we have $\|\mathbf{S}_u(\xi)-\mathbf{S}_u(\eta) \|_{\mrm{TV}} \leq 2\|\xi-\eta\|$.
\end{proposition}

\begin{proof}[Proof of Proposition \ref{prop.teeveee}] 
Let $\lambda_1,\ldots,\lambda_d \in \mathbb{T}$ be the eigenvalues of $u$ with associated eigenvectors $\vartheta_1,\ldots,\vartheta_d$. Write $\xi = \alpha_1 \vartheta_1 + \cdots + \alpha_d \vartheta_d$ and $\eta =  \beta_1 \vartheta_1 + \cdots + \beta_d \vartheta_d$ for $\alpha_1,\ldots,\alpha_d,\beta_1,\ldots,\beta_d \in \mathbb{C}$. Letting $p_\xi$ and $p_\eta$ be the projections onto $\spn(\xi)$ and $\spn(\eta)$ respectively, we have $|\alpha_k| = \|p_\xi \vartheta_k\|$ and $|\beta_k| = \|p_\eta \vartheta_k\|$ for all $k \in \{1,\ldots,d\}$ and so:  \[ \mathbf{S}_u(\xi) = |\alpha_1|^2 \boldsymbol{\delta}_{\lambda_1} + \cdots +  |\alpha_d|^2 \boldsymbol{\delta}_{\lambda_d} \qqquad  \mathbf{S}_u(\eta) = |\beta_1|^2 \boldsymbol{\delta}_{\lambda_1} + \cdots +  |\beta_d|^2 \boldsymbol{\delta}_{\lambda_d} \] We also have 
\begin{align*} 
\|\xi-\eta\| & = \|(\alpha_1-\beta_1)\vartheta_1 + \cdots + (\alpha_d - \beta_d) \vartheta_d\| \\
& = \sqrt{|\alpha_1-\beta_1|^2 + \cdots + |\alpha_d-\beta_d|^2} \\
& \geq \sqrt{(|\alpha_1|-|\beta_1|)^2 + \cdots + (|\alpha_d|-|\beta_d|)^2}   
\end{align*}
Now, we can regard the vectors $(|\alpha_1|^2,\ldots,|\alpha_d|^2)$ and $(|\beta_1|^2,\ldots,|\beta_d|^2)$ as probability measures on the set $\{1,\ldots,d\}$. In this context the inequality between the Hellinger distance and the total variation distance (see Item 8 in Section 3 of \cite{419a6ad1-1c79-3c86-94d2-f62a47a8c518}) reads: 
\[
2\sqrt{(|\alpha_1|-|\beta_1|)^2 + \cdots + (|\alpha_d|-|\beta_d|)^2} \geq \left \|\alpha_1|^2-|\beta_1|^2 \right| + \cdots +\left\|\alpha_d|^2-|\beta_d|^2\right|
\]
Proposition \ref{prop.teeveee} now follows by combining the two previous displays.
\end{proof}

\begin{proposition} \label{prop.street} Let $n \in \mathbb{N}$, let $\eta \leq 1/n$ and let $\xi_1,\ldots,\xi_n$ be a list of unit vectors in a finite dimensional Hilbert space $\mathcal{H}$ such that such that $|\langle \xi_j,\xi_k \rangle| \leq \eta$ for all distinct pairs $j,k \in \{1,\ldots,n\}$. Then for any unitary $u \in \mrm{Un}(\mathcal{H})$ we have: 
\[
\nml \mathbf{S}_u(\xi_1,\ldots,\xi_n) - \frac{\mathbf{S}_u(\xi_1)+\cdots+\mathbf{S}_u(\xi_n)}{n} \nmr_{\mrm{TV}} \leq  2\eta n 
\] 
\end{proposition}

\begin{proof}[Proof of Proposition \ref{prop.street}] 
According to Proposition \ref{prop.orthogonal}, there exists an orthonormal family of vectors $\vartheta_1,\ldots,\vartheta_n$ such that $\|\xi_k - \vartheta_k\| \leq \eta n$ and such that $\spn(\xi_1,\ldots,\xi_n) = \spn(\vartheta_1,\ldots,\vartheta_n)$. In particular, the last assertion implies:
\[
\mathbf{S}_u(\xi_1,\ldots,\xi_n) = \mathbf{S}_u(\vartheta_1,\ldots,\vartheta_n)
\]

Using Proposition \ref{prop.teeveee} we find $\|\mathbf{S}_u(\xi_k)-\mathbf{S}_u(\vartheta_k)\|_{\mrm{TV}} \leq 2\eta n$ for all $k \in \{1,\ldots,n\}$. Therefore we have \[ \nml \frac{\mathbf{S}_u(\xi_1)+\cdots+\mathbf{S}_u(\xi_n)}{n} - \frac{\mathbf{S}_u(\vartheta_1)+\cdots+\mathbf{S}_u(\vartheta_n)}{n} \nmr_{\mrm{TV}} \leq 2\eta n  \] Since the family $\vartheta_1,\ldots,\vartheta_n$ is orthonormal, Proposition \ref{prop.subadd} implies: \[ \mathbf{S}_u(\vartheta_1,\ldots,\vartheta_n) = \frac{\mathbf{S}_u(\vartheta_1)+\cdots+\mathbf{S}_u(\vartheta_n)}{n} \] Proposition \ref{prop.street} now follows by combining the last three displays. \end{proof} 

From Propositions \ref{prop.teevee} and \ref{prop.street} we obtain the following.

\begin{corollary}  Let $\xi_1,\ldots,\xi_n$ be a family of unit vectors in a finite dimensional Hilbert space $\mathcal{H}$ such that $|\langle \xi_j,\xi_k \rangle| \leq \eta$ for all distinct pairs $j,k \in \{1,\ldots,n\}$. \label{cor.str-1} Then for any unitary $u \in \mrm{Un}(\mathcal{H})$ we have: \[\dss_N(\mathbf{S}_u(\xi_1,\ldots,\xi_n)) \leq  \frac{\dss_N(\mathbf{S}_u(\xi_1)) + \cdots + \dss_N(\mathbf{S}_u(\xi_n))}{n} +  2\eta n \] \end{corollary}

\subsection{Complementary projections and disuniformity estimate}

The statement of Proposition \ref{prop.persist} is somewhat technical, but essentially it gives us control over the extent to which disuniformity can decay on the orthogonal complement of an enlarged projection.

\begin{proposition} \label{prop.persist} 
Let $\delta > 0$ and $n \in \mathbb{N}$. Let $\mathcal{H}$ be a Hilbert space with finite dimension $d \in \mathbb{N}$, let $u \in \mrm{Un}(\mathcal{H})$ be a unitary operator and let $p \in \mrm{Proj}(\mathcal{H})$ be a projection. Also let $\xi_1,\ldots,\xi_n$ be a list of vectors in $\mathcal{H}$. We make the following assumptions about these objects. 
\begin{align} \label{eq.hypo-0} 
\frac{1}{n} & \geq 2 \delta \\ 
\frac{d}{2} & \geq \tr(p) \label{eq.hypo-1} \\
\label{eq.hypo-2} \frac{1}{10}  & \geq \max_{1 \leq k \leq n} \|p \xi_k\| \\  
\label{eq.hypo-3}  \delta & \geq \max_{1 \leq j,k \leq n:j \neq k} |\langle (I_\mathcal{H}-p)\xi_j, \xi_k \rangle|
\end{align} 
Then letting $q$ be the projection onto $\spn(p\mathcal{H},\xi_1,\ldots,\xi_n)$ we have the estimate below. 
\begin{equation} \label{eq.today-2} 
\dss_N(\mathbf{S}_u(I_\mathcal{H}-q)) \leq \left(1+ \frac{4n}{d}\right)\left( \dss_N(\mathbf{S}_u(I_\mathcal{H}-p)) + \frac{2}{d}\left( 4\delta n^2 + \sum_{k=1}^n \dss_N\bigl (\mathbf{S}_u((I_\mathcal{H}-p)\xi_k) \bigr)  \right) \right)
\end{equation}  
\end{proposition}

\begin{proof}[Proof of Proposition \ref{prop.persist}] 
For $k \in \{1,\ldots,n\}$ write
\[
\varphi_k = \frac{(I_\mathcal{H}-p)\xi_k}{\sqrt{1-\|p\xi_k\|^2}} 
\]
so that $\varphi_1,\ldots,\varphi_k$ are unit vectors.  Using our hypotheses \eqref{eq.hypo-0} and \eqref{eq.hypo-3}, Clause (b) of Proposition \ref{prop.compspan} implies the following. 
\[ 
\spn(p\mathcal{H},\xi_1,\ldots,\xi_n) = p\mathcal{H} \oplus \spn(\varphi_1,\ldots,\varphi_n)  
\]
and so writing $\mathcal{V} = (I_\mathcal{H}-p)\mathcal{H}$ and $\mathcal{W} = \spn(\varphi_1,\ldots,\varphi_n)$ we have: 
\[
\mathcal{H} = p\mathcal{H} \oplus \mathcal{W} \oplus (\mathcal{V} \ominus \mathcal{W}) 
\]
It follows that $I_\mathcal{H}-q$ is the projection onto $\mathcal{V} \ominus \mathcal{W}$ and so with $c = \dm(\mathcal{W})/\dm(\mathcal{V})$, Proposition \ref{prop.subtract2} implies: 
\[
\dss_N(\mathbf{S}_u(I_\mathcal{H}-q)) \leq \frac{\dss_N(\mathbf{S}_u((I_\mathcal{H}-p))) + c\dss_N(\mathbf{S}_u(\varphi_1,\ldots,\varphi_n)) }{1-c} 
\]
Combining \eqref{eq.hypo-0} and \eqref{eq.hypo-3} with Clause (a) of Proposition \ref{prop.compspan} we find $\dm(\mathcal{W}) =n$ and using the hypothesis \eqref{eq.hypo-1} we have $\dm(\mathcal{V}) = d - \tr(p) \geq d/2$. Therefore $c \leq 2n/d$ and so the previous display gives 
\begin{equation} \label{eq.ff-1} 
\dss_N(\mathbf{S}_u(I_\mathcal{H}-q)) \leq \left(1+\frac{4n}{d}\right)  \left(\dss_N(\mathbf{S}_u(I_\mathcal{H}-p)) + \frac{2n\dss_N(\mathbf{S}_u(\varphi_1,\ldots,\varphi_n))}{d} \right) 
\end{equation}

Now, using the assumptions \eqref{eq.hypo-2} and \eqref{eq.hypo-3} we obtain the following for all distinct pairs $j,k \in \{1,\ldots,n\}$.  
\[
|\langle \varphi_j, \varphi_k \rangle| = \frac{|\langle (I_\mathcal{H}-p)\xi_j, \xi_k \rangle|}{\sqrt{(1-\|p\xi_j\|^2)(1-\|p\xi_k\|^2)}}  \leq 2 \delta 
\]
From the previous display and Corollary \ref{cor.str-1} we find:
\begin{equation} \label{eq.ff-2}
\dss_N(\mathbf{S}_u(\varphi_1,\ldots,\varphi_n)) \leq  \frac{\dss_N(\mathbf{S}_u(\varphi_1)) + \cdots + \dss_N(\mathbf{S}_u(\varphi_n))}{n} +  4\delta n 
\end{equation} 
Since $\spn(\varphi_k) = \spn((I_\mathcal{H}-p)\xi_k)$ for all $k \in \{1,\ldots,n\}$ we have $\mathbf{S}_u(\varphi_k) = \mathbf{S}((I_\mathcal{H}-p)\xi_k)$. By combining the last assertion with \eqref{eq.ff-1} and \eqref{eq.ff-2} we obtain the required inequality \eqref{eq.today-2}. 
\end{proof}

\section{Concentration of measure estimates} \label{sec.concest}

\subsection{Unitary and projection concentration estimate}

Proposition \ref{prop.ject} ensures that projections with relatively small trace typically shrink vectors.

\begin{proposition} \label{prop.ject} 
Let $G$ be a countable group, let $\alpha:G \to \mrm{Un}(\mathcal{H})$ be a map and let $p \in \mrm{Proj}(\mathcal{H})$ be a projection such that $\tr(p) \leq \kappa^2 \dm(\mathcal{H})/2$. Then for any finite subset $F$ of $G$ we have: 
\[ 
\sigma_\mathcal{H}\left( \left\{ \xi \in \mathbb{S}_\mathcal{H}: \max_{g \in F} \|p\alpha(g)\xi\| \leq \kappa \right\} \right) \geq 1- 2|F|\exp \left(-\frac{\kappa^4 (2 \dm(\mathcal{H})-1)}{32} \right ) 
\]	 
\end{proposition}

\begin{proof}[Proof of Proposition \ref{prop.ject}] 
Fix $g \in F$ and note that: \[ \frac{\|p \alpha(g)\|_{\mrm{HS}}^2}{\dm (\mathcal{H})} \leq  \frac{\|p \|_{\mrm{HS}}^2 ||\alpha(g)||_{\mrm{op}}^2}{\dm (\mathcal{H})} =  \frac{\|p\|_{\mrm{HS}}^2}{\dm( \mathcal{H})} \leq \frac{\kappa^2}{2} \] Choosing $c=\kappa^2/2 $ in Corollary \ref{cor.dev-1} we find 
\[ 
\sigma_\mathcal{H}\left( \left\{ \xi \in \mathbb{S}_\mathcal{H}: \|p\alpha(g)\xi\|^2 \leq \kappa^2  \right\} \right) \geq 1- 2\exp \left(-\frac{\kappa^4(2 \dm(\mathcal{H})-1)}{32} \right ) 
\] 
and so Proposition \ref{prop.ject} follows from the complementary union bound. \end{proof}

\begin{proposition} \label{prop.ooo} Let $G$ be a countable group, let $\nu,\kappa > 0$, let $\alpha:G \to \mrm{Un}(\mathcal{H})$ be a map and let $E$ and $F$ be finite subsets of $G$. Let $p \in \mrm{Proj}(\mathcal{H})$ be a projection such that $\tr(p) \leq \kappa^2 \dm(\mathcal{H})/2$ and such that $\|(I-p)\alpha(g)p\|^2_{\mrm{HS}} \leq \nu^2 \tr(p)$ for all $g \in E$. Then we have: 
\[ 
\sigma_\mathcal{H}\left( \left\{ \xi \in \mathbb{S}_\mathcal{H}: \max_{g \in E, h \in F} \|(I_\mathcal{H}-p)\alpha(g)p \alpha(h)\xi\| \leq \kappa \nu \right\} \right) \geq 1- 2|E||F|\exp \left(-\frac{\kappa^4\nu^4(2 \dm(\mathcal{H})-1)}{32} \right ) 
\] 
\end{proposition}

\begin{proof}[Proof of Proposition \ref{prop.ooo}] Let $g \in E$ and $h \in F$. We have \[ \|(I_\mathcal{H}-p)\alpha(g)p \alpha(h)\|_{\mrm{HS}}^2 = \|(I_\mathcal{H}-p)\alpha(g)p\|_{\mrm{HS}}^2 \leq \frac{\kappa^2\nu^2 \dm(\mathcal{H})}{2} \] where the left equality follows from unitary invariance of the Hilbert-Schmidt norm and the right inequality follows from out hypotheses. Setting $c = \kappa^2 \nu^2/2$ in Corollary \ref{cor.dev-1} and applying the complementary union bound we obtain the result. 
\end{proof}

\subsection{Hyperlinear concentration estimate}

Proposition \ref{prop.conc-0} ensures that the unitaries in a hyperlinear approximation act like homomorphisms on most vectors.

\begin{proposition} \label{prop.conc-0}  Let $G$ be a countable group and let $\alpha: G \to \mrm{Un}(\mathcal{H})$ be an $(F,\delta)$-hyperlinear approximation to $G$. Then we have: 
\[ 
\sigma_\mathcal{H} \left(\left \{ \xi \in \mathbb{S}_\mathcal{H}: \max_{g,h \in F} \|(\alpha(g)\alpha(h) - \alpha(gh))\xi\| \leq 2 \delta \right \} \right ) \geq 1- 2|F|^2\exp \left(-\frac{\delta^4(2 \dm(\mathcal{H})-1)}{128} \right )   
\] 
\end{proposition}

\begin{proof}[Proof of Proposition \ref{prop.conc-0}] Since $\alpha$ was assumed to be an $(F,\delta)$-hyperlinear approximation to $G$, for $g,h \in F$ we have $||\alpha(g)\alpha(h) - \alpha(gh)||_{\mrm{HS}} \leq \delta \sqrt{\mrm{dm}(\mathcal{H})}$. We also observe that: \[ \|s\|_{\mrm{op}}\leq \|\alpha(g) \alpha(h)\|_{\mrm{op}}+\|\alpha(gh)\|_{\mrm{op}}=2 \] Thus we can apply Corollary \ref{cor.dev-1} with $s=\alpha(g)\alpha(h) - \alpha(gh)$ and $c = \delta^2$ to obtain: \[ 
\sigma_\mathcal{H} \left(\left \{ \xi \in \mathbb{S}_\mathcal{H}: \max_{g,h \in F} \|(\alpha(g)\alpha(h) - \alpha(gh))\xi\| \leq 2 \delta \right \} \right ) \geq 1- 2 \exp \left(-\frac{\delta^4(2 \dm(\mathcal{H})-1)}{128} \right )   
\] Proposition \ref{prop.conc-0} now follows by applying the complementary union bound to the previous inequality for all pairs $g,h \in F$.

\end{proof}

\subsection{Disuniformity concentration estimate}

Proposition \ref{prop.conc-1} allows us to force individual vectors to have similar disuniformity estimates to the entire space.

\begin{proposition} \label{prop.conc-1} Let $N \in \mathbb{N}$ and let $p \in \mrm{Proj}(\mathcal{H})$ be a projection such that $\tr(p) \geq \nu \dm(\mathcal{H})$ for some $\nu > 0$ and let $\delta > 0$ satisfy $\delta \leq \nu/2$. Then for any unitary operator $u \in \mrm{Un}(\mathcal{H})$ we have:  
\[ 
\sigma_\mathcal{H} \left(\left \{ \xi \in \mathbb{S}_\mathcal{H}: \mtt{DS}_N(\mathbf{S}_u(p\xi)) \leq \mtt{DS}_N(\mathbf{S}_u(p)) + \frac{5 \delta N}{\nu^2} \right \} \right ) \geq 1- 4N \exp \left(-\frac{\delta^2(2 \dm(\mathcal{H})-1)}{8} \right )  
\] 
\end{proposition}

\begin{proof}[Proof of Proposition \ref{prop.conc-1}]

Write $d = \dm(\mathcal{H})$ and let $\lambda_1,\ldots,\lambda_d \in \mathbb{T}$ be the eigenvalues of $u$ with corresponding orthonormal eigenvectors $\vartheta_1,\ldots,\vartheta_d$. For $k \in \{0,\ldots,N-1\}$ write $\gamma_k$ for $[\mathbf{S}_u(p)](\mathscr{I}_N[k])$ and write $q_k$ for the projection onto $\spn( \{\vartheta_\ell: \lambda_\ell \in \mathscr{I}_N[k] \})$. We make the following claim.

\begin{claim} \label{cla.qual} 
Assume that a unit vector $\xi \in \mathcal{H}$ satisfies the following for some $k \in \{0,\ldots,N-1\}$. 
\begin{equation} \label{eq.qual-3} 
\|q_kp \xi\|^2 \approx_\delta \frac{\gamma_k \tr(p)}{d}  \end{equation} 
Also assume that $\|p\xi\|^2 \approx_\delta \tr(p)/d$. Then we have $\|q_k p \xi\|^2/\|p\xi\|^2 \approx_{5 \delta /\nu^2} \gamma_k$. 
\end{claim}

\begin{proof}[Proof of Claim \ref{cla.qual}] 
We first observe 
\[
\|p\xi\|^2 \geq \frac{\tr(p)}{d} - \delta \geq \frac{\nu}{2} 
\] 
where the right inequality holds since the hypotheses of Proposition \ref{prop.conc-1} include the assertions $\tr(p) \geq \nu d$ and $\delta \leq \nu/2$. 
The mean value theorem gives 
\[
\left \vert \frac{1}{x}-\frac{1}{y} \right \vert \leq \frac{4|x-y|}{\nu^2} 
\]
whenever $\min\{x,y\}\geq \nu/2$. So we obtain: 
\begin{equation} \label{eq.qual-4} 
\left \vert \frac{1}{\|p\xi\|^2} - \frac{d}{\tr(p)} \right \vert  \leq \frac{4}{\nu^2}\left\vert \|p\xi\|^2-\frac{\tr(p)}{d}  \right\vert \leq \frac{4 \delta}{\nu^2} 
\end{equation}
We also have
\begin{equation} \label{eq.qual-5} 
\left \vert \frac{\|q_kp\xi\|^2 d }{\tr(p)} - \gamma_k  \right \vert \leq \frac{d \delta}{\tr(p)} \leq \frac{\delta}{\nu}
\end{equation} 
where the left inequality above follows from our hypothesis in \eqref{eq.qual-3}.  We now compute: 
\begin{align} 
\left \vert \frac{\|q_kp\xi\|^2}{\|p \xi\|^2} - \gamma_k \right \vert & \leq  \left \vert \frac{\|q_kp\xi\|^2}{\|p \xi\|^2} -  \frac{\|q_kp\xi\|^2 d }{\tr(p)}  \right \vert + \left \vert \frac{\|q_kp\xi\|^2 d }{\tr(p)} - \gamma_k  \right \vert  \nonumber \\ & = \|q_kp\xi\|^2 \left \vert \frac{1}{\|p \xi\|^2} -  \frac{d}{\tr(p)}  \right \vert + \left \vert \frac{\|q_kp\xi\|^2 d }{\tr(p)} - \gamma_k  \right \vert \label{eq.qual-5.5} \\ & \leq \|q_kp\xi\|^2 \left \vert \frac{1}{\|p \xi\|^2} -  \frac{d}{\tr(p)}  \right \vert + \frac{\delta}{\nu} \label{eq.qual-6}  \\ &  \leq \|q_kp\xi\|^2 \frac{4 \delta}{\nu^2} + \frac{\delta}{\nu}  \label{eq.qual-7}  \\ & \leq \frac{5\delta}{\nu^2} \label{eq.qual-8}  
\end{align} 
Here, 
\begin{itemize} 
\item \eqref{eq.qual-6} follows from \eqref{eq.qual-5.5} by \eqref{eq.qual-5}, 
\item \eqref{eq.qual-7} follows from \eqref{eq.qual-6} by \eqref{eq.qual-4}, 
\item and \eqref{eq.qual-8} follows from \eqref{eq.qual-7} since $\xi$ is a unit vector and $\|pq\|_{\mrm{op}} \leq 1$.  
\end{itemize}
This completes the proof of Claim \ref{cla.qual}. 
\end{proof} 

We now make a second claim.

\begin{claim} \label{cla.qual-2}
Assume that a unit vector $\xi \in \mathcal{H}$ satisfies the following for all $k \in \{0,\ldots,N-1\}$.
\[  
\|q_kp \xi\|^2 \approx_\delta \frac{\gamma_k \tr(p)}{d}  
\] 
Again assume that $\|p\xi\|^2 \approx_\delta \tr(p)/d$. Then we have 
\[
\mtt{DS}_N(\mathbf{S}_u(p\xi)) \leq \mtt{DS}_N(\mathbf{S}_u(p)) + \frac{5 \delta N}{\nu^2} 
\] 
\end{claim}

\begin{proof}[Proof of Claim \ref{cla.qual-2}] Using Claim \ref{cla.qual}, we find 
\[ 
\max_{0 \leq k \leq N-1} \left \vert \frac{\|q_kp \xi\|^2}{\|p \xi\|^2} - \gamma_k \right \vert \leq \frac{5 \delta}{\nu^2} 
\]
Applying Proposition \ref{prop.qual-1} we find  
\[
\max_{0 \leq k \leq N-1} \Bigl \vert [\mathbf{S}_u(p\xi)](\mathscr{I}_N[k]) - \gamma_k \Bigr \vert \leq \frac{5 \delta}{\nu^2} 
\]
and unwrapping the definition of $\gamma_k$ we obtain: 
\[
\max_{0 \leq k \leq N-1} \Bigl \vert [\mathbf{S}_u(p\xi)](\mathscr{I}_N[k]) - [\mathbf{S}_u(p)](\mathscr{I}_N[k]) \Bigr \vert \leq \frac{5 \delta}{\nu^2} 
\]
Therefore: 
\begin{align*}
\Bigl \vert \mtt{DS}_N(\mathbf{S}_u(p\xi)) - \mtt{DS}_N(\mathbf{S}_u(p)) \Bigr \vert &= \left \vert \left( \sum_{k=0}^{N-1} \left \vert [\mathbf{S}_u(p\xi)](\mathscr{I}_N[k]) - \frac{1}{N} \right \vert \right)- \left( \sum_{k=0}^{N-1} \left \vert [\mathbf{S}_u(p)](\mathscr{I}_N[k]) - \frac{1}{N} \right \vert \right) \right \vert \\
& \leq \sum_{k=0}^{N-1} \Bigl \vert [\mathbf{S}_u(p\xi)](\mathscr{I}_N[k])- [\mathbf{S}_u(p)](\mathscr{I}_N[k]) \Bigr \vert \\
& \leq \frac{5 \delta N}{\nu^2} 
\end{align*} 
This completes the proof of Claim \ref{cla.qual-2}.
\end{proof}

According to Proposition \ref{prop.interval}, for all $k \in \{0,\ldots,N-1\}$ we have $\|q_kp\|_{\mrm{HS}}^2 = \gamma_k\tr(p)$. Using Proposition \ref{prop.dev-1}, we obtain  
\[
\sigma_\mathcal{H} \left ( \left \{ \xi \in \mathbb{S}_\mathcal{H}: \|q_k p \xi\|^2 \approx_\delta \frac{\gamma_k\tr(p)}{d} \right \} \right ) \geq 1- 2\exp \left(-\frac{\delta^2(2 \dm(\mathcal{H})-1)}{8 } \right ) 
\]
and thus the complementary union bound implies: 
\[ 
\sigma_\mathcal{H} \left ( \left \{ \xi \in \mathbb{S}_\mathcal{H}: \|q_k p \xi\|^2 \approx_\delta \frac{\gamma_k\tr(p)}{d} \mbox{ for all }k \in \{0,\ldots,N-1\} \right \} \right ) \geq 1- 2N\exp \left(-\frac{\delta^2(2 \dm(\mathcal{H})-1)}{8 } \right ) 
\]

Again using Proposition \ref{prop.dev-1} we find: \[ \sigma_\mathcal{H} \left ( \left \{ \xi \in \mathbb{S}_\mathcal{H}: \|p \xi\|^2 \approx_\delta \frac{\tr(p)}{d} \right \} \right ) \geq 1- 2\exp \left(-\frac{\delta^2(2 \dm(\mathcal{H})-1)}{8 } \right )  \] Combining the last two displays, we see \[ \sigma_\mathcal{H} \left ( \Bigl \{ \xi \in \mathbb{S}_\mathcal{H}: \xi \mbox{ satisfies the hypotheses of Claim }\ref{cla.qual-2} \Bigr \} \right ) \geq 1- 2(N+1)\exp \left(-\frac{\delta^2(2 \dm(\mathcal{H})-1)}{8} \right )  \] and so we obtain Proposition \ref{prop.conc-1} by simplifying $2(N+1)$ to $4N$. \end{proof}

\begin{corollary} \label{cor.conc-1} Let $\alpha:G \to \mrm{Un}(\mathcal{H})$ be a map, let $F$ be a finite subset of $G$ and let $\eta,\nu > 0$ and $N \in \mathbb{N}$. Let $p \in \mrm{Proj}(\mathcal{H})$ be a projection such that $\tr(p) \geq \nu \dm(\mathcal{H})$. Also assume that $\mtt{DS}_N(\mathbf{S}_{\alpha(g)}(p)) \leq \eta$ for all nontrivial $g \in F$. Then we have:  
\begin{equation} \label{eq.ovv-2} \sigma_\mathcal{H} \left(\left \{ \xi \in \mathbb{S}_\mathcal{H}: \max_{g,h \in F:g \neq e_G} \mtt{DS}_N(\mathbf{S}_{\alpha(g)}(p\alpha(h)\xi)) \leq 2 \eta \right \} \right ) \geq 1- 4N|F| \exp \left(-\frac{\eta^2\nu^4(2 \dm(\mathcal{H})-1)}{200N^2} \right ) 
\end{equation} \end{corollary}

\begin{proof}[Proof of Corollary \ref{cor.conc-1}] For a fixed $g \in F$, we apply Proposition \ref{prop.conc-1} with $\delta = \eta \nu^2/(5N)$ to find:  \[ \sigma_\mathcal{H} \Bigl(\Bigl \{ \xi \in \mathbb{S}_\mathcal{H}: \mtt{DS}_N(\mathbf{S}_u(p\xi)) \leq \mtt{DS}_N(\mathbf{S}_{\alpha(g)}(p)) + \eta  \Bigr \} \Bigr ) \geq 1- 4N \exp \left(-\frac{\eta^2\nu^4(2 \dm(\mathcal{H})-1)}{200N^2} \right )  \] Since we have assumed $\mtt{DS}_N(\mathbf{S}_{\alpha(g)}(p)) \leq \eta$ for all nontrivial $g \in F$, from the previous display along with the complementary union bound we obtain the following. 
\begin{equation} \label{eq.ovv}
\sigma_\mathcal{H} \left(\left \{ \xi \in \mathbb{S}_\mathcal{H}: \max_{g \in F:g \neq e_G} \mtt{DS}_N(\mathbf{S}_{\alpha(g)}(p\xi)) \leq 2 \eta \right \} \right ) \geq 1- 4N|F| \exp \left(-\frac{\eta^2\nu^4(2 \dm(\mathcal{H})-1)}{200N^2} \right )  
\end{equation}
On the other hand, for any unitary operator $u \in \mrm{Un}(\mathcal{H})$ we can use the $u$-invariance of $\sigma_\mathcal{H}$ to find: 
\[
\sigma_\mathcal{H} \left(\left \{ \xi \in \mathbb{S}_\mathcal{H}: \max_{g \in F} \mtt{DS}_N(\mathbf{S}_{\alpha(g)}(p \xi)) \leq 2 \eta \right \} \right )  = \sigma_\mathcal{H} \left(\left \{ \xi \in \mathbb{S}_\mathcal{H}: \max_{g \in F} \mtt{DS}_N(\mathbf{S}_{\alpha(g)}(p u \xi)) \leq 2 \eta \right \} \right ) 
\]
Letting $u$ range over $\{\alpha(g):g \in F\}$ and combining the previous display with \eqref{eq.ovv} we obtain the required inequality \eqref{eq.ovv-2}.   
\end{proof}

\subsection{Projection orthogonality concentration estimate} \label{sec.projconc}

Proposition \ref{prop.conc-22} shows that disuniformity ensures typical projections are almost orthogonal.

\begin{proposition} \label{prop.conc-22} Let $\mathcal{H}$ be a finite dimensional Hilbert space, let $p \in \mrm{Proj}(\mathcal{H})$ be a projection and let $u,v \in \mrm{Un}(\mathcal{H})$ be unitary operators. Then we have 
\[
\sigma_\mathcal{H}\left (\left \{ \xi \in \mathbb{S}_\mathcal{H}: |\langle pu\xi, pv \xi  \rangle| \leq \mtt{DS}(\mathbf{S}_{uv^\ast}(p)) + \frac{3}{N} \right \} \right ) \geq  1- 2\exp \left(-\frac{2 \dm(\mathcal{H})-1}{8N^2 } \right ) 
\] 
\end{proposition}

\begin{proof}[Proof of Proposition \ref{prop.conc-22}] We have 
\[ 
\frac{|\tr(v^\ast pu)|}{\dm(\mathcal{H})} = \frac{|\tr(puv^\ast)|}{\dm(\mathcal{H})} \leq  \frac{|\tr(puv^\ast)|}{\tr(p)} \leq \mtt{DS}_N(\mathbf{S}_{uv^\ast}(p)) + \frac{2}{N}
\]
where the rightmost inequality above follows from Proposition \ref{prop.trace}. Moreover, we have $|\langle v^\ast p u \xi, \xi \rangle| = |\langle p u \xi, pv \xi \rangle|$ and so Proposition \ref{prop.conc-22} follows by combining the previous display the case $c = 1/N$ of Proposition \ref{prop.dev-1}. 
\end{proof}

By applying the complementary union bound we obtain the following corollary of Proposition \ref{prop.conc-22}.

\begin{corollary} \label{prop.conc-2} 
Let $G$ be a countable group, let $\mathcal{H}$ be a finite dimensional Hilbert space, let $\alpha:G \to \mrm{Un}(\mathcal{H})$ be a map and let $F$ be a finite subset of $G$. Let $\eta > 0$ and $N \in \mathbb{N}$ satisfy $3/N \leq \eta$. Let $p \in \mrm{Proj}(\mathcal{H})$ be a projection and let $u \in \mrm{Un}(\mathcal{H})$ be a unitary operator. Assume that $\mtt{DS}_N(\mathbf{S}_{\alpha(g)}(p)) \leq \eta$ for all nontrivial $g \in F$. Then we have:  
\[ 
\sigma_\mathcal{H}\left (\left \{ \xi \in \mathbb{S}_\mathcal{H}: \max_{g,h \in F: g \neq h} |\langle p\alpha(g)\xi, p\alpha(h) \xi  \rangle| \leq 2 \eta  \right \} \right ) \geq  1- 2|F|^2\exp \left(-\frac{2 \dm(\mathcal{H})-1}{8N^2 } \right ) 
\]
\end{corollary}

\begin{remark} \label{rem.torsion-2} Corollary \ref{prop.conc-2} is where the modifications described in Remark \ref{rem.torsion-1} become relevant. If the group $G$ has torsion, we may have to forgo the interpretation of $\mtt{DS}(\mathbf{S}_{uv^\ast}(p))$ as in Proposition \ref{prop.conc-22} as being small, and so we cannot use it to obtain a small upper bound on $|\langle p\alpha(g)\xi, p\alpha(h) \xi  \rangle|$. However, if $g \in G$ has order $n$ and we reinterpret $\mtt{DS}_N(\mathbf{S}_{\alpha(g)}(p))$ as level-$N$ distance to the uniform measure on the $n^{\mrm{th}}$ roots of unity, we can again conclude that the desired orthogonality holds with high probability. Here the key fact being used is simply that  with $\mu$ being uniform on the $n^{\mrm{th}}$ roots of unity we have \[ \int_\mathbb{T} z^k \deee \mu(z) = 0 \] for all $k \in \mathbb{Z}$ with $k \notin n \mathbb{Z}$. \end{remark}

\section{Probabalistic construction lemmas} \label{sec.ainv}

\subsection{Almost invariant projections in hyperlinear approximation}

Lemma \ref{prop.ainv} combines many of the estimates from Section \ref{sec.concest} to obtain a conclusion about almost invariance of extended projections. The key feature here is that the bound $\nu$ in the conclusion is exactly the same $\nu$ as in the hypothesis, which will allow us to apply Lemma \ref{prop.ainv} recursively.

\begin{lemma} \label{prop.ainv} Let $G$ be a countable group, let $\mathcal{H}$ be a finite dimensional Hilbert space and let $\alpha:G \to \mrm{Un}(\mathcal{H})$ be a map. Let $\nu \in (0,1/12)$, let $E$ and $F$ be finite subsets of $G$ with $E \subseteq F$ and let $|EF \triangle F| \leq \nu^2|F|/2$. Let $\delta > 0$ satisfy $2\delta |F| \leq \nu/6$ and let $p \in \mrm{Proj}(\mathcal{H})$ be a projection such that the following holds. 
\begin{equation} \label{eq.oooo}
\max_{g \in E} \|(I_\mathcal{H}-p)\alpha(g)p\|_{\mrm{HS}} \leq \nu \sqrt{\tr(p)} 
\end{equation} 
Also let $(\xi_g)_{g \in F}$ be a family of unit vectors in $\mathcal{H}$ with the following properties. \begin{description} 
\item[(a)] For all $g \in F$ we have $\|p\xi_g\| \leq 1/10 $. 
\item[(b)] For all $g \in E$ and all $h \in F$ we have $\|(I-p)\alpha(g)p\xi_h\| \leq \nu/5$. 
\item[(c)] For all $g,h \in F$ with $gh \in F$ we have $\|\alpha(g)\xi_h - \xi_{gh}\| \leq \delta$.   
\item[(d)] For all distinct pairs $g,h \in F$ we have $|\langle (I_\mathcal{H}-p)\xi_g, \xi_h \rangle| \leq \delta$. \end{description} 
Then writing $q$ for the projection onto $\spn(p\mathcal{H},\{\xi_g:g \in F\})$, we have $\|(I-q)\alpha(g)q\|_{\mrm{HS}} \leq \nu \sqrt{\tr(q)}$ for all $g \in E$.
\end{lemma}

\begin{proof}[Proof of Lemma \ref{prop.ainv}]
For $h \in F$ define: 
\begin{equation} \label{eq.air-0} 
\varphi_h = \frac{(I_\mathcal{H}-p)\xi_h}{\sqrt{1-\|p\xi_h\|^2}}  
\end{equation}
We now fix an element $g \in E$ through the remainder of the proof of Lemma \ref{prop.ainv}. Assume $h \in g^{-1} F \cap F$. We calculate 
\begin{align}
\|q\alpha(g) \xi_h\|^2 & = \sup \Bigl \{ |\langle \alpha(g)\xi_h,\eta \rangle| : \eta \in q\mathcal{H} \mbox{ is a unit vector}  \Bigr\} \label{eq.air-2}  \\
& \geq |\langle \alpha(g) \xi_h,\xi_{gh} \rangle| \label{eq.air-3} \\ 
& = 1-\frac{\|\alpha(g)\xi_h-\xi_{gh}\|^2}{2}  \label{eq.air-4} \\
& \geq 1-\delta^2 \label{eq.air-5} 
\end{align}
where \eqref{eq.air-3} follows from \eqref{eq.air-2} since the assumption that $gh \in F$ implies $\xi_{gh} \in q \mathcal{H}$ and \eqref{eq.air-5} follows from \eqref{eq.air-4} by Hypothesis (c). Continuing, we compute 
\begin{align}
\|(I_\mathcal{H}-q)\alpha(g)\varphi_h\| & = \frac{\|(I_\mathcal{H}-q)\alpha(g)(I_\mathcal{H}-p)\xi_h\|}{\sqrt{1-\|p\xi_h\|^2}} \label{eq.air-9} \\
& \leq \frac{10}{\sqrt{99}} \|(I_\mathcal{H}-q)\alpha(g)(I_\mathcal{H}-p)\xi_h\| \label{eq.air-10} \\
& \leq  \frac{10}{\sqrt{99}} \Bigl( \|(I_\mathcal{H}-q)\alpha(g)p\xi_h\|+\|(I_\mathcal{H}-q)\alpha(g)\xi_h\| \Bigr) \label{eq.air-11} \\ 
& \leq \frac{10}{\sqrt{99}} \Bigl( \|(I_\mathcal{H}-p)\alpha(g)p\xi_h\|+\|(I_\mathcal{H}-q)\alpha(g)\xi_h\| \Bigr) \label{eq.air-12} \\ 
& \leq \frac{10}{\sqrt{99}} \Bigl( \frac{\nu}{5} +\|(I_\mathcal{H}-q)\alpha(g)\xi_h\| \Bigr) \label{eq.air-13} \\
& \leq \frac{10}{\sqrt{99}} \Bigl( \frac{\nu}{5} +\delta \Bigr) \label{eq.air-14} \\ 
& \leq \frac{\nu}{3} \label{eq.air-15} 
\end{align} 
This computation may be justified as follows. 
\begin{itemize} 
\item The equality in \eqref{eq.air-9} follows from the definition of $\varphi_h$ in \eqref{eq.air-0}. 
\item \eqref{eq.air-10} follows from \eqref{eq.air-9} by Hypothesis (a). 
\item \eqref{eq.air-12} follows from \eqref{eq.air-11} since $p \leq q$. 
\item \eqref{eq.air-13} follows from \eqref{eq.air-12} by Hypothesis (b).
\item \eqref{eq.air-14} follows from \eqref{eq.air-13} by \eqref{eq.air-5}. 
\item \eqref{eq.air-15} follows from \eqref{eq.air-14} since we have assumed $\delta \leq \nu/12$ and numerical computation shows: \[ \frac{10}{\sqrt{99}} \left( \frac{1}{5}+\frac{1}{12} \right) \leq \frac{1}{3} \]
\end{itemize} 
Now, using Hypothesis (d) of the present proposition together with Clause (c) of Proposition \ref{prop.compspan} we can find an orthonormal basis $\{\vartheta_g\}_{g \in F}$ for $\spn(\{\vartheta_g\}_{g \in F})$ such that:
\begin{equation} \label{eq.lib-1}
\max_{g \in F} \|\vartheta_g - \varphi_g\| \leq 2\delta |F| 
\end{equation} 
Thus for $h \in g^{-1}F \cap F$ we have:
\begin{align} \label{eq.jay-1} \|(I_\mathcal{H}-q)\alpha(g)\vartheta_h\| & \leq \|(I_\mathcal{H}-q)\alpha(g)(\vartheta_h-\varphi_h)\|  +  \|(I_\mathcal{H}-q)\alpha(g)\varphi_h\|  \\ & \leq \|\vartheta_h-\varphi_h\|  +  \|(I_\mathcal{H}-q)\alpha(g)\varphi_h\|  \label{eq.jay-2}   \\ & \leq  2\delta |F| +  \|(I_\mathcal{H}-q)\alpha(g)\vartheta_h\| \label{eq.jay-3}  \\ & \leq 2 \delta |F| + \frac{\nu}{3} \label{eq.jay-4}  \\ & \leq  \frac{\nu}{2} \label{eq.jay-5}  \end{align}

Here, 

\begin{itemize} \item (\ref{eq.jay-2}) follows from (\ref{eq.jay-1}) since we have \begin{equation} \label{eq.opp} ||(I_\mathcal{H} - q)\alpha(g)||_{\mrm{op}} \leq ||(I_\mathcal{H} - q)||_{\mrm{op}} ||\alpha(g)||_{\mrm{op}} \leq 1 \end{equation} \item (\ref{eq.jay-3}) follows from (\ref{eq.jay-2}) by (\ref{eq.lib-1}), \item (\ref{eq.jay-4}) follows from (\ref{eq.jay-3}) by the calculation ending in (\ref{eq.air-15}), \item and (\ref{eq.jay-5}) follows from (\ref{eq.jay-4}) since we have assumed $2 \delta |F| \leq \nu/6$. \end{itemize}

Write $n = \tr(p)$ and let $\vartheta_1,\ldots,\vartheta_n$ be an orthonormal basis for $p\mathcal{H}$. For $g \in E$ we compute: 
\begin{align} \|(I_\mathcal{H}-q)\alpha(g)q\|_{\mrm{HS}}^2 & = \left( \sum_{k=1}^n \|(I_\mathcal{H}-q)\alpha(g)\vartheta_k\|^2 \right)+\left( \sum_{h \in F} \|(I_\mathcal{H}-q)\alpha(g)\vartheta_h\|^2 \right) \label{eq.fuck-1} \\
& \leq  \left( \sum_{k=1}^n \|(I_\mathcal{H}-p)\alpha(g)\vartheta_k\|^2 \right)+\left( \sum_{h \in F} \|(I_\mathcal{H}-q)\alpha(g)\vartheta_h\|^2 \right) \label{eq.fuck-2} \\
& =  \|(I_\mathcal{H}-p)\alpha(g)p\|_{\mrm{HS}}^2 +\left( \sum_{h \in F} \|(I_\mathcal{H}-q)\alpha(g)\vartheta_h\|^2 \right) \label{eq.fuck-3} \\
& \leq  \nu^2 \tr(p)  +\left( \sum_{h \in F} \|(I_\mathcal{H}-q)\alpha(g)\vartheta_h\|^2 \right) \label{eq.fuck-4}  
\end{align} 
Here, 
\begin{itemize} 
\item \eqref{eq.fuck-2} follows from \eqref{eq.fuck-1} since $p \leq q$, 
\item \eqref{eq.fuck-3} follows from \eqref{eq.fuck-2} since $\vartheta_1,\ldots,\vartheta_n$ forms an orthonormal basis for $p\mathcal{H}$, 
\item and \eqref{eq.fuck-4} follows from \eqref{eq.fuck-3} by our assumption in \eqref{eq.oooo}. 
\end{itemize}
Now, we compute again:  \begin{align}   \sum_{h \in F} \|(I_\mathcal{H}-q)\alpha(g)\vartheta_h\|^2 & =  \left( \sum_{h \in F \setminus g^{-1}F }  \|(I_\mathcal{H}-q)\alpha(g)\vartheta_h\|^2 \right) + \left( \sum_{h \in g^{-1}F \cap F}  \|(I_\mathcal{H}-q)\alpha(g)\vartheta_h\|^2 \right) \label{eq.vv-1} 
\\ & \leq |F \setminus g^{-1}F| + \left( \sum_{h \in g^{-1}F \cap F}  \|(I_\mathcal{H}-q)\alpha(g)\vartheta_h\|^2 \right) \label{eq.vv-2} 
\\ & \leq \frac{\nu^2 |F|}{2} + \left( \sum_{h \in g^{-1}F \cap F}  \|(I_\mathcal{H}-q)\alpha(g)\vartheta_h\|^2 \right)  \label{eq.vv-3} 
\\ & \leq \frac{\nu^2 |F|}{2} + \frac{\nu^2|g^{-1}F \cap F|}{4} \label{eq.vv-4}
\\ & \leq \nu^2|F| \label{eq.vv-5}   
\end{align} 
This computation may be justified as follows. 
\begin{itemize} 
\item \eqref{eq.vv-2} follows from \eqref{eq.vv-1} since (\ref{eq.opp}) implies $\|(I_\mathcal{H}-q)\alpha(g)\vartheta_h\|_{\mrm{op}} \leq 1$ for all $h \in F$. 
\item \eqref{eq.vv-3} follows from \eqref{eq.vv-2} since we have assumed $|EF \triangle F| \leq \nu^2|F|/2$. 
\item \eqref{eq.vv-4} follows from \eqref{eq.vv-3} by the calculation ending in (\ref{eq.jay-5}). 
\end{itemize}
Now, combining \eqref{eq.fuck-4} with \eqref{eq.vv-5} we obtain: 
\[ 
\|(I_\mathcal{H}-q)\alpha(g)q\|_{\mrm{HS}}^2 \leq \nu^2 (\tr(p)+|F|) 
\]
From Hypothesis (d) of the present proposition and Clause (a) of Proposition \ref{prop.compspan} we find $\tr(q) = \tr(p) + |F|$, and so the proof of Lemma \ref{prop.ainv} is complete. 
\end{proof}

\subsection{Recursion lemma}  \label{sec.lem}

Lemma \ref{lem.1} below is the main technical tool for the `inner recursion' that will be used to prove Lemma \ref{main.lemma}. \begin{remark} \label{rem.loop} As with Lemma \ref{prop.ainv}, the key feature of Lemma \ref{lem.1} is that the parameter $\nu$ in Clause (b) of its conclusion is the same as the parameter $\nu$ in the hypothesis. Similarly, the level of decay in disuniformity appearing in Clause (c) of the conclusion is normalized by $\dm(\mathcal{H})$, which will allow us to recursively apply Lemma \ref{lem.1} roughly $\kappa \dm(\mathcal{H})$ times in the proof of Lemma \ref{main.lemma}. The price we pay is that the statement of Lemma \ref{lem.1} involves a variety of somewhat unmotivated conditions, most of which are `internal' to the recursion and can be dispensed with at the final stage.\end{remark}

\begin{lemma} \label{lem.1} Let $G$ be a countable group and let $0<\kappa< 1/10$ and $\nu>0 $ be small fixed parameters. Let $E,F$ be finite symmetric subsets of $G$ with $E^2 \subseteq F$ and $|EF \triangle F| \leq \nu^2|F|/4$. Also let $\delta>0$ and $N \in \mathbb{N}$  be two numbers that satisfy $3/N \leq \delta$ and $2\delta |F| \leq \nu/6<1/2$. Then there exists $D \in \mathbb{N}$ such that if $\mathcal{H}$ is a finite-dimensional Hilbert space with $\dm(\mathcal{H}) \geq D$ and $\alpha:G \to \mrm{Un}(\mathcal{H})$ is an $(F,\delta)$-hyperlinear approximation then the following holds. Assume that $p \in \mrm{Proj}(\mathcal{H})$ is a projection with $\tr(p) \leq \kappa^2 \dm(\mathcal{H})/2$ and such that 
\begin{equation} 
\label{eq.hill-2} 
\max_{g \in E} \|(I_\mathcal{H}-p)\alpha(g)p\|_{\mrm{HS}}^2 \leq \nu^2 \tr(p) 
\end{equation} 
Also assume that 
\begin{equation} \label{eq.hill-11} 
\max_{g \in F:g \neq e_G} \mtt{DS}_N(\mathbf{S}_{\alpha(g)}(I_\mathcal{H}-p)) \leq \delta 
\end{equation} 
Then there exist unit vectors $\{\xi_g:g \in F\}$ such that if we write $q$ for the projection onto $\spn(p\mathcal{H},\{\xi_g:g \in F\})$ then the following statements are true. 
\begin{description} 
\item[(a)] We have $\|\alpha(g)\xi_h - \xi_{gh}\| \leq 2 \delta$ for all $g,h \in F$. 
\item[(b)] We have $\|(I_\mathcal{H}-q)\alpha(g)q\|_{\mrm{HS}} \leq \nu \sqrt{\tr(q)}$ for all $g \in E$. 
\item[(c)] We have 
\[ 
\mtt{DS}_N(\mathbf{S}_{\alpha(g)}(I_\mathcal{H}-q)) \leq \delta \exp\left( \frac{16|F|^2}{\dm(\mathcal{H})} \right) 
\]
for all nontrivial $g\in F$.
\item[(d)] We have $ \|p\xi_g\|\leq \kappa$ for all $g \in F$.
\end{description}
Also assume that there exists an $(E,\nu)$-sofic induced hyperlinear approximation $\beta: G \to \mrm{Un}(p\mathcal{H})$ such that for every $g\in E$ we have: \begin{equation} \label{eq.tri-0} \|( \beta(g) - \alpha(g))p\|_{\mrm{HS}}^2 \leq (\nu^2 + (4\kappa+10|F|\delta)^2)\tr(p) \end{equation} Then there exists an $(E, \nu)$-sofic induced hyperlinear approximation $\tau:G \to \mrm{Un}(q\mathcal{H})$  such that the following holds for every $g \in E$:
\begin{equation} \label{eq.trii}
\|( \tau(g) - \alpha(g))q\|_{\mrm{HS}}^2 \leq (\nu^2 + (4\kappa+10|F|\delta)^2) \tr(q)
\end{equation}
\end{lemma}

\begin{proof}[Proof of Lemma \ref{lem.1}] Let $\mathcal{H}$ be a finite dimensional Hilbert space, whose dimension we temporarily leave unstipulated, and let $\alpha:G \to \mrm{Un}(\mathcal{H})$ be an $(F,\delta)$-hyperlinear approximation to $G$. Since we have assumed $\tr(p) \leq \kappa^2 \dm(\mathcal{H})/2$, we can use Proposition \ref{prop.ject} to obtain: 
\begin{equation} \label{eq.light-1}
\sigma_\mathcal{H}\left( \left\{ \xi \in \mathbb{S}_\mathcal{H}: \max_{g \in F} \|p\alpha(g)\xi\| \leq \kappa \right\} \right) \geq 1- 2|F|\exp \left(-\frac{\kappa^4(2 \dm(\mathcal{H})-1)}{32} \right )  
\end{equation} 
Combining our hypothesis in (\ref{eq.hill-2}) with Proposition \ref{prop.ooo} we find: \begin{equation} \label{eq.light-2} \sigma_\mathcal{H}\left( \left\{ \xi \in \mathbb{S}_\mathcal{H}: \max_{g \in E,h \in F} \|(I_\mathcal{H}-p)\alpha(g)p \alpha(h) \xi\| \leq \kappa \nu \right\} \right) \geq 1- 2|F|\exp \left(-\frac{\kappa^4\nu^4(2 \dm(\mathcal{H})-1)}{32} \right ) 
\end{equation} 
Since $\alpha$ was assumed to be an $(F,\delta)$-hyperlinear approximation to $G$, using Proposition \ref{prop.conc-0} we find: \begin{equation}
\sigma_\mathcal{H}\Bigl( \Bigl\{ \xi \in \mathbb{S}_\mathcal{H}: \max_{g,h \in F} \|\alpha(gh)\xi - \alpha(g)\alpha(h)\xi\| \leq 2 \delta \Bigr\} \Bigr) \geq 1- 2|F|^2\exp \left(-\frac{\delta^2(2 \dm(\mathcal{H})-1)}{128} \right ) \label{eq.light-3} 
\end{equation}
By combining the assumption in (\ref{eq.hill-11}) with Corollary \ref{prop.conc-2} we obtain the following. 

\begin{equation}
\sigma_\mathcal{H}\left (\left \{ \xi \in \mathbb{S}_\mathcal{H}: \max_{g,h \in F:g \neq h} \bigl \vert \bigl \langle (I_\mathcal{H}- p) \alpha(g) \xi, \alpha(h) \xi  \bigr \rangle \bigr \vert \leq 2 \delta \right \} \right ) \geq 1- 2|F|^2\exp \left(-\frac{2 \dm(\mathcal{H})-1}{8N^2 } \right ) \label{eq.light-5}
\end{equation} 
Again using the assumption in (\ref{eq.hill-11}),  Corollary \ref{cor.conc-1} applied to $I_\mathcal{H}-p$ implies the following.\begin{align}  \sigma_\mathcal{H} \left(\left \{ \xi \in \mathbb{S}_\mathcal{H}: \max_{g,h \in F:g \neq e_G} \mtt{DS}_N(\mathbf{S}_{\alpha(g)}((I_\mathcal{H}-p)\alpha(h)\xi)) \leq 2 \delta \right \} \right ) & \nonumber \\ &\hspace{-2 cm} \geq 1- 4N|F| \exp \left(-\frac{\delta^2\nu^4(2 \dm(\mathcal{H})-1)}{200N^2} \right ) \label{eq.light-4} 
\end{align} 
Now, choose $D \in \mathbb{N}$ such that if $\mathcal{H}$ is a finite dimensional Hilbert space with $\dm(\mathcal{H}) \geq D$ then the intersection of the five sets on the left of \eqref{eq.light-1} - \eqref{eq.light-4} is nonempty. In any such space, we can find $\xi \in \mathbb{S}_\mathcal{H}$ such that the following hold.
\begin{align} 
\label{eq.hilt-0}  \kappa & \geq \max_{g \in F} \|p\alpha(g)\xi\| \\ 
\label{eq.hilt-0.5} \kappa \nu & \geq \max_{g \in E,h \in F} \|(I-p)\alpha(g)p\alpha(h) \xi\| \\ 
\label{eq.hilt-1} 2 \delta & \geq \max_{g,h \in F} \|\alpha(gh)\xi - \alpha(g)\alpha(h)\xi\| \\ 
2 \delta & \geq \max_{g,h \in F:g \neq h} \bigl \vert \bigl \langle (I_\mathcal{H}- p) \alpha(g) \xi, (I_\mathcal{H}- p) \alpha(h) \xi  \bigr \rangle \bigr \vert  \label{eq.hilt-3} \\
2 \delta & \geq  \max_{g,h \in F:g \neq e_G} \mtt{DS}_N\bigl(\mathbf{S}_{\alpha(g)}((I_\mathcal{H}-p)\alpha(h)\xi)\bigr) \label{eq.hilt-4}  
\end{align}

Choose $\xi_g = \alpha(g)\xi$. Then Clause (a)  and Clause (d) in the conclusion of Lemma \ref{lem.1} follow directly from  \eqref{eq.hilt-1} and \eqref{eq.hilt-0} respectively. We also see that \eqref{eq.hilt-0}, \eqref{eq.hilt-0.5}, \eqref{eq.hilt-1} and \eqref{eq.hilt-3} coincide respectively with Hypotheses (a), (b), (c) and (d) of Lemma \ref{prop.ainv}. Therefore Clause (b) in the conclusion of Lemma \ref{lem.1} follows from the conclusion of Proposition \ref{prop.ainv}. \\
\\
Next we establish Clause (c) in the conclusion of Lemma \ref{lem.1}. To this end, enumerate $F = \{g_1,\ldots,g_n\}$ and write for simplicity $\xi_k = \xi_{g_k}$. Also fix a nontrivial element $g \in F$ and write $u$ for $\alpha(g)$. With these notations we may reinterpret the assertions in \eqref{eq.hilt-0} and \eqref{eq.hilt-3} respectively as: 
\[ 
\frac{1}{10} \geq \max_{1 \leq k \leq n} \|p \xi_k\| \qqquad \qqquad  2 \delta  \geq \max_{1 \leq j,k \leq n:j \neq k} \bigl \vert \bigl \langle (I_\mathcal{H}- p) \xi_j, (I_\mathcal{H}- p) \xi_k  \bigr \rangle \bigr \vert  
\] 
The inequalities in the previous display matche the hypotheses \eqref{eq.hypo-2} and \eqref{eq.hypo-3} of Proposition \ref{prop.persist}, and so we conclude: 
\[  
\dss_N(\mathbf{S}_u(I_\mathcal{H}-q)) \leq \left(1+ \frac{4|F|}{d}\right)\left( \dss_N(\mathbf{S}_u(I_\mathcal{H}-p)) + \frac{2}{d}\left( 4\delta |F|^2 + \sum_{k=1}^n \dss_N\bigl (\mathbf{S}_u((I_\mathcal{H}-p)\xi_k) \bigr)  \right) \right)  
\]
 
Using \eqref{eq.hilt-4}  we get the following estimate
\[   
2 \delta  \geq  \max_{1 \leq k \leq n} \mtt{DS}_N(\mathbf{S}_u((I_\mathcal{H}-p)\xi_k))  
\]
and combining the last two displays we find:

\[
\dss_N(\mathbf{S}_u(I_\mathcal{H}-q)) \leq \left(1+ \frac{4|F|}{d}\right)\left( \dss_N(\mathbf{S}_u(I_\mathcal{H}-p)) + \frac{8 \delta |F|^2 + 4|F|\delta}{d}  \right) 
\] 
Now, using our assumption in \eqref{eq.hill-11} we see the previous display implies: 
\[ \dss_N(\mathbf{S}_u(I_\mathcal{H}-q)) \leq \delta \left(1+ \frac{4|F|}{d}\right)\left(1 + \frac{8|F|^2 + 4|F|}{d}  \right)  
\] 
Using the elementary inequality $1+x \leq \exp(x)$ we find: 
\[ 
\left(1+ \frac{4|F|}{d}\right)\left(1 + \frac{8|F|^2 + 4|F|}{d}  \right) \leq \exp\left( \frac{16|F|^2}{d} \right)
\] 

This completes the proof of Clause (c) in Lemma \ref{lem.1}. We now address the claim at the end of Lemma \ref{lem.1} involving extension of sofic-induced approximations. Let $\beta: G \to \mrm{Un}(p\mathcal{H})$ be as in the statement of Lemma \ref{lem.1}, with the goal of defining the extension $\tau$. We let $\tau(g)=\beta(g)$ on $p\mathcal{H}$ and endeavor to define  $\tau(g)$ on $q\mathcal{H} \ominus p\mathcal{H}$.\\
\\
For $k \in \{1,\ldots,n\}$ we define: \[ \varphi_k = \frac{(I_\mathcal{H} - p) \xi_k }{\sqrt{1-||p\xi_k||^2}} \] The inequality in \eqref{eq.hilt-3} allows us to apply Proposition \ref{prop.compspan} to the vectors $\{\varphi_k\}_{k=1}^n$ with $\delta$ replaced by $2\delta$, so we obtain an orthonormal basis $\{ \vartheta_k \}_{k=1}^n$ for $q \mathcal{H} \ominus p \mathcal{H}$ such that $\| \vartheta_k -\varphi_k\| \leq 4\delta n$ for all $k \in \{1,\ldots,n\}$. We also observe that Clause (d) in the hypotheses of Lemma \ref{lem.1} implies  $ \| \varphi_k-\xi_k\| \leq 2\kappa$ for all $k \in \{1,\ldots,n\}$, and so we have: \[ \max\limits_{1\leq k\leq n}\| \vartheta_k-\xi_k\| \leq 4\delta n+2\kappa \] It will now be convenient to undo the change in indexing of our vectors from $F$ to $\{1,\ldots,n\}$, so we return to labelling $\xi_1,\ldots,\xi_n$ as $\{\xi_g\}_{g \in F}$ and introduce corresponding labels $\{\vartheta_g\}_{g \in F}$. Combining the previous display with Clause (a) in the hypotheses of Lemma \ref{lem.1}, we find that the following holds for all pairs $g,h \in F$ with $gh \in F$.
\begin{equation} \label{eq.tri-2}
\|\alpha(g)\vartheta_h-\vartheta_{gh}\|\leq 10\delta n+4\kappa
\end{equation}

Now, fix $g \in E $. For $h \in F \cap g^{-1} F$, we define $\tau(g)\vartheta_h=\vartheta_{gh} $ and then arbitrarily extend $\tau(g)$ to be a permutation of the vectors $\{\vartheta_g\}_{g \in F}$. Since these vectors are orthonormal and we have assumed $|EF \triangle F| \leq \nu^2|F|/4$, it follows that $\tau$ is induced from an $(E,\nu)$-sofic approximation to $G$. Thus in order to complete the proof of Lemma \ref{lem.1} it remains to verify the inequality in (\ref{eq.trii}). To this end, fix $g \in E$ and suppose $h \in F \cap g^{-1}F$. Then we find: \begin{equation} \label{eq.trii-2} ||(\tau(g) - \alpha(g)) \vartheta_h ||  \leq ||\tau(g) \vartheta_h - \vartheta_{gh}|| + ||\vartheta_{gh} - \alpha(g) \vartheta_h|| = ||\vartheta_{gh} - \alpha(g) \vartheta_h|| \leq 10\delta|F| +4\kappa \end{equation} Here, the equality in the middle follows since we stipulated that $\tau(g)\vartheta_h = \vartheta_{gh}$ and the inequality on the right follows from (\ref{eq.tri-2}). For $g \in E$ we compute: \begin{align} \|( \tau(g) - \alpha(g))(q-p)\|_{\mrm{HS}}^2 & = \sum_{h \in F} ||(\tau(g) - \alpha(g)) \vartheta_h||^2 \label{eq.tri-5} \\ & \leq \sum_{h \in F \setminus g^{-1}F} \Bigl( ||(\tau(g) - \alpha(g)) \vartheta_h||^2 \Bigr) + \sum_{h \in F \cap g^{-1}F} \Bigl( ||(\tau(g) - \alpha(g)) \vartheta_h||^2 \Bigr) \label{eq.tri-6}  \\ & \leq 4|F \setminus g^{-1}F| + \sum_{h \in F \cap g^{-1}F} ||(\tau(g) - \alpha(g)) \vartheta_h||^2 \label{eq.tri-7}  \\ & \leq \nu^2|F| + \sum_{h \in F \cap g^{-1}F} ||(\tau(g) - \alpha(g)) \vartheta_h||^2 \label{eq.tri-8}  \\ & \leq  (\nu^2 + (10\delta|F| +4\kappa)^2)|F| \label{eq.hoot}  \end{align} This computation can be justified as follows.

\begin{itemize} \item The equality in (\ref{eq.tri-5}) follows from the fact that $\{\vartheta_h\}_{h \in F}$ is an orthonormal basis for $q \mathcal{H} \ominus p \mathcal{H}$. \item (\ref{eq.tri-7}) follows from (\ref{eq.tri-6}) since we have $||\tau(g) - \alpha(g)||_{\mrm{op}} \leq 2$. \item (\ref{eq.tri-8}) follows from (\ref{eq.tri-7}) since we have assumed that $|EF \triangle F| \leq \nu^2|F|/4$. \item (\ref{eq.hoot}) follows from (\ref{eq.tri-8}) by (\ref{eq.trii-2}). \end{itemize} Thus we obtain: \begin{align} \label{eq.tri-9} \|( \tau(g) - \alpha(g))q\|_{\mrm{HS}}^2 & = \|( \tau(g) - \alpha(g))(q-p)\|_{\mrm{HS}}^2 + \|( \tau(g) - \alpha(g))p\|_{\mrm{HS}}^2 \\ & \label{eq.tri-10} \leq (\nu^2 + (10\delta|F| +4\kappa)^2)|F| +\|( \tau(g) - \alpha(g))p\|_{\mrm{HS}}^2 \\ & \label{eq.tri-11} \leq  (\nu^2 + (10\delta|F| +4\kappa)^2)(|F|+\tr(p))  \end{align} Here, \begin{itemize} \item the equality in (\ref{eq.tri-9}) follows from the fact that the operators on the right are orthogonal in the trace inner product, \item (\ref{eq.tri-10}) follows from (\ref{eq.tri-9}) by the calculation ending in (\ref{eq.hoot}), \item and (\ref{eq.tri-11}) follows from (\ref{eq.tri-10}) since the hypothesis of Lemma \ref{lem.1}.   \end{itemize} Using Clauses (a) and (b) in Proposition \ref{prop.compspan} along with (\ref{eq.hilt-3}) we find that $\tr(q) = |F| + \tr(p)$. Combining this with the calculation ending in (\ref{eq.tri-11}) we obtain the desired inequality (\ref{eq.trii}). This completes the proof of Lemma \ref{lem.1}. \end{proof}

\section{Main recursions}  \label{main}

\subsection{Proof of Lemma \ref{main.lemma}}

In this section we prove Lemma \ref{main.lemma} by recursive applications of Lemma \ref{lem.1}. It is easy to see that it suffices to prove lemma for symmetric subsets $E$, so we will operate under this assumption throughout the proof.

\begin{proof}[Proof of Lemma \ref{main.lemma}]
Let $0 < \kappa < 1/10$ and $\nu > 0$ and let  $E\subset G$ be a finite subset of $G$. Choose a finite subset $F$ of $G$, $N\in \mathbb{N}$ and $\delta>0$ that satisfy the conditions of Lemma \ref{lem.1}. We also assume that $10|F|\delta \leq \kappa$. Let $D$ be the lower bound on the dimension of a Hilbert space provided by Lemma \ref{lem.1} relative to these parameters. Also define $\delta_0 = \delta \exp(-16|F|)$. Using Proposition \ref{prop.ds-trace}, we can find $M \in \mathbb{N}$ and $\delta_1 > 0$ such that if $\mathcal{H}$ is a finite dimensional Hilbert space and $u \in \mrm{Un}(\mathcal{H})$ satisfies $|\mrm{tr}(u^k)| \leq \delta_1$ for all nonzero $k \in \{-M,\ldots,M\}$ then we have:\begin{equation} \label{eq.hill-0} 
 \mtt{DS}_N(\mathbf{S}_{u}(\mathcal{H})) \leq \delta_0 
\end{equation} Now, choose a finite subset $F_\bullet$ of $G$ and $\delta_\bullet > 0$ such that the following hold. \begin{description} \item[(i)] We have $g^k \in F_\bullet$ for all $g \in F$ and all $k \in \{-M,\ldots,M\}$. \item[(ii)] For any $(F_\bullet,\delta_\bullet)$-hyperlinear approximation $\alpha:G\to \mrm{Un}(\mathcal{H})$ we have $\mrm{dm}(\mathcal{H})>D$. \end{description} Now, let $\mathcal{H}$ be a finite dimensional Hilbert space and let $\alpha:G \to \mrm{Un}(\mathcal{H})$ be an $(F_\bullet,\delta_\bullet)$-hyperlinear approximation to $G$. We now execute a recursive procedure, iteratively applying Lemma \ref{lem.1} to build an increasing sequence of projections on $\mathcal{H}$. \begin{description} \item[(Base step)] We start with $p_0 = 0$. It is immediate that (\ref{eq.hill-2}) holds, and (\ref{eq.hill-11}) follows from (\ref{eq.hill-0}) with $\delta = \delta_0$. Thus Lemma \ref{lem.1} provides a projection $q = p_1$ and a sofic-induced hyperlinear approximation $\beta_1:G \to \mrm{Un}(p_1\mathcal{H})$ satisfying the conclusion of Lemma \ref{lem.1}. \item[(Recursive step)] Assume we have a projection $p_n$ on $\mathcal{H}$ such that $n|F| = \tr(p_n) \leq \kappa^2 \dm(\mathcal{H})/2$ and such that (\ref{eq.hill-2}) is satisfied and (\ref{eq.hill-11}) is satisfied with $\delta_0 \exp(16n|F|^2/\dm(\mathcal{H}))$ instead of $\delta$. Then we can apply Lemma \ref{lem.1} to obtain a projection $p_{n+1}$ and a sofic-induced hyperlinear approximation $\beta_1:G \to \mrm{Un}(p_1\mathcal{H})$ satisfying the conclusion of Lemma \ref{lem.1}. Taking into account Remark \ref{rem.loop}, we can continue the recursion at the sole cost of the bound in (\ref{eq.hill-2}) decaying to: \[ \max_{g \in F \setminus \{e_G\}} \mtt{DS}_N(\mathbf{S}_{\alpha(g)}(I_\mathcal{H}-p_{n+1})) \leq \delta_0 \exp \left (\frac{16(n+1)|F|^2}{\dm(\mathcal{H})} \right) \leq \delta \] \end{description}

Now, this procedure terminates at the first step when $\tr(p_n) >\kappa^2\dm(\mathcal{H})/2$. At this stage we have constructed an $(E,\nu)$-sofic induced hyperlinear approximation $\beta_n:G \to \mrm{Un}(p_n\mathcal{H})$ that satisfies
\[ 
\max_{g\in E}\|(\beta_n(g)-\alpha(g))q\|^2_{\mrm{HS}}\leq (\nu^2+ (4\kappa+10|F|\delta)^2)\tr(p_n) \leq (\nu^2+5\kappa)\tr(p_n)
\]

Having completed the recursion, we take $p=p_n$, $\mathcal{V}=p\mathcal{H}$, $\mathcal{U}= \mathcal{H} \ominus p \mathcal{H}$ and $\beta=\beta_n$. It follows directly from the previous display that $\beta$ satisfies (\ref{mainlemma1}). \\
\\

On the other hand, symmetry of $E$ and  Clause (b) in the conclusion of Lemma \ref{lem.1}  imply that the projection $I_{\mathcal{H}}-p$ is also almost invariant with respect to the action of $\alpha(E)$. Therefore,  Proposition \ref{prop.unitcom} implies that for every $g\in E$ there is an operator $\gamma(g)$ that satisfies $\gamma(g)^*\gamma(g)=\gamma(g)\gamma(g)^*=q$ and $\|(\gamma(g)-\alpha(g))(I_{\mathcal{H}}-p)\|_{\mrm{HS}}^2\leq 4\nu^2\tr(I_{\mathcal{H}}-p)$. Then $\gamma:G\to \mrm{Un}(\mathcal{U})$ is an $(E,6\nu+\delta)$-hyperlinear approximation to $G$ satisfying (\ref{mainlemma2}). These objects are as required by the conclusion of Lemma \ref{main.lemma} and so its proof is complete. \end{proof}

\subsection{Proof of Theorem \ref{thm.sofic}}

In this section we prove Theorem \ref{thm.sofic} by recursive applications of Lemma \ref{main.lemma}.

\begin{proof}[Proof of Theorem \ref{thm.sofic}]
As in the statement of Theorem \ref{thm.sofic}, let $G$ be an amenable group, let $E$ be a finite subset of $G$ and let $\epsilon > 0$. We let $\kappa = \epsilon/25$ and choose $m\in\mathbb{N}$ such that: 
\begin{equation} \left(1-\frac{\kappa^2}{2} \right)^m \leq \frac{\epsilon}{2} \end{equation} By an $m$-fold iteration of Lemma \ref{main.lemma}, we can choose a decreasing sequence $(\nu_n)_{n=1}^m$ of positive numbers such that \begin{equation} \sum\limits_{n= 1}^m \nu_n^2 \leq  \frac{\epsilon^2}{16m}\left(1-\frac{\kappa^2}{2} \right)^m \label{eq.sum}  \end{equation} along with an increasing sequence $(F_n)_{n=1}^m$ of finite subsets of $G$ such that $F_1 = E$ and such that for all $n \in \{1,\ldots,m\}$ we have that every $(F_{n+1},7\nu_{n+1})$-hyperlinear approximation to $G$ admits a decomposition into an $(F_n,\nu_n)$-sofic induced and an $(F_n,7\nu_n)$-hyperlinear approximations as in the conclusion of Lemma \ref{main.lemma}.\\
\\
We claim that we can take $F=F_m$ and $\delta=\nu_m$ to satisfy the conclusion of Theorem \ref{thm.sofic}. To verify this claim, consider a finite-dimensional Hilbert space $\mathcal{H}$ and an $(F,\delta)$ hyperlinear approximation $\alpha:G\to \mrm{Un}(\mathcal{H})$. By applying Lemma \ref{main.lemma} to $\alpha$ we obtain a decomposition $\mathcal{H}=\mathcal{U}_{m-1}\oplus\mathcal{V}_{m-1}$ along with an $(F_{m-1},\nu_{m-1})$-sofic induced hyperlinear approximation $\beta_{m-1}:G\to \mrm{Un}(\mathcal{V}_{m-1})$ and an $(F_{m-1},7\nu_{m-1})$ hyperlinear approximation $\gamma_{m-1}:G\to \mrm{Un}(\mathcal{U}_{m-1})$ such that if we let $p_{m-1}$ and $q_{m-1}$ be the projections onto $\mathcal{V}_{m-1}$ and $\mathcal{U}_{m-1}$ respectively then we have:
\begin{equation*} 
    \|(\alpha(g)-\beta_{m-1}(g))p_{m-1}\|_{\mrm{HS}}^2\leq (5\kappa+\nu_{m-1})^2\dm(\mathcal{V}_{m-1})
\end{equation*}
and
\begin{equation*}
    \|(\alpha(g)-\gamma_{m-1}(g))q_{m-1}\|_{\mrm{HS}}^2\leq 4\nu_{m-1}^2\dm(\mathcal{U}_{m-1}).
\end{equation*}

Similarly, for every $n \in \{1,\ldots,m-1\}$ we obtain a decomposition $\mathcal{U}_n=\mathcal{U}_{n-1}\oplus\mathcal{V}_{n-1}$ along with an $(F_{n-1},\nu_{n-1})$-sofic induced hyperlinear approximation $\beta_{n-1}:G\to \mrm{Un}(\mathcal{V}_{n-1})$ and an $(F_{n-1},7\nu_{n-1})$ hyperlinear approximation $\gamma_{n-1}:G\to \mrm{Un}(\mathcal{U}_{n-1})$ such that if we let $p_{n-1}$ and $q_{n-1}$ be the projections onto $\mathcal{V}_{n-1}$ and $\mathcal{U}_{n-1}$ respectively then we have:

\begin{equation} \label{eq.ouch-2}
    \|(\gamma_n(g)-\beta_{n-1}(g))p_{n-1}\|_{\mrm{HS}}^2\leq (5\kappa+\nu_{n-1})^2\dm(\mathcal{V}_{n-1})
\end{equation}
and:
\begin{equation}
    \|(\gamma_n(g)-\gamma_{n-1}(g))q_{n-1}\|_{\mrm{HS}}^2\leq 4\nu_{n-1}^2\dm(\mathcal{U}_{n-1}) \label{eq.ouch-1}
\end{equation}

At the conclusion of this process, the space $\mathcal{H}$ is decomposed as
\begin{equation*}
    \mathcal{H}=\mathcal{V}_{m-1}\oplus \mathcal{V}_{m-2}\oplus \cdots\oplus\mathcal{V}_0\oplus \mathcal{U}_0.
\end{equation*}

where:

\begin{equation} \label{eq.chain}
    \dm(\mathcal{U}_0)\leq \left(1-\frac{\kappa^2}2\right) \dm(\mathcal{V}_1)\leq \left(1-\frac{\kappa^2}2\right)^2 \dm(\mathcal{V}_2) \leq \dots \leq \left(1-\frac{\kappa^2}2\right)^m\dm(\mathcal{H})\leq \frac{\epsilon}2 \dm(\mathcal{H})
\end{equation}

For $g\in G$ define $\beta(g) \in \mrm{Un}(\mathcal{H})$ by setting \[ \beta(g)=\beta_{m-1}(g)\oplus\dots\oplus \beta_0(g) \oplus I_{\mathcal{U}_0} \] In other words, $\beta(g)$ acts as $\beta_n(g)$ on $\mathcal{V}_n$ for  $n \in \{0,\dots, m-1 \}$, and $\beta(g)$ acts trivially on $\mathcal{U}_0$. Then, since each $\beta_n$ is an $(E,\nu_1)$-sofic induced hyperlinear approximation  and \[
\nu_1\left( 1-\frac{\dm(\mathcal{U}_0)}{\dm(\mathcal{H})} \right)^{-1} \leq \nu_1 \left( 1-\frac{\epsilon}{2}\right)^{-1} \leq \frac{\nu_1}{2} \leq \epsilon \] we have that  $\beta$ is an $(E,\epsilon)$-sofic induced hyperlinear approximation to $G$. Thus it remains to show that for all $g \in E$ we have: \begin{equation} \label{eq.want} ||\alpha(g) - \beta(g)||_{\mrm{HS}}^2 \leq \epsilon^2 \mrm{dm}(\mathcal{H}) \end{equation} To this end, fix $g \in E$ along with $n \in \{0, \dots, m-1\} $. By construction, we have $q_kp_n=p_n$ for all $k \in \{n,\ldots,m-1\}$ and therefore: \begin{equation} \label{eq.gee-0} 
(\alpha(g)-\beta_n(g))p_n \ = (\alpha(g)-\gamma_{m-1}(g))q_{m-1}+(\gamma_{m-1}(g)-\gamma_{m-2}(g))q_{m-2}+\dots +(\gamma_{n+1}(g)-\beta_n(g))p_n \end{equation} Adopting the notation $\alpha(g)=\gamma_m(g)$, we find:  \begin{align}
    \|(\alpha(g)-\beta_n(g))p_n\|_{\mrm{HS}} &\leq \| (\gamma_{n+1}(g)-\beta_i(g))p_n\|_{\mrm{HS}}+ \sum_{k=n+1}^{m-1} \| (\gamma_{k+1}(g)-\gamma_{k}(g))q_{k}\|_{\mrm{HS}} \label{eq.geed-1} \\
    &\leq  (5\kappa+\nu_n)\sqrt{\dm(\mathcal{V}_{n})} + \sum_{k=n+1}^{m-1} 2\nu_k \sqrt{\dm(\mathcal{U}_{n})} \label{eq.geed-2} \\
    &\leq\left(5\kappa+\nu_{n} + 2\left(1-\frac{\kappa^2}{2} \right)^{-m/2} \sum_{k=i+1}^{m-1} \nu_k \right)\sqrt{\dm(\mathcal{V}_{n}) } \label{eq.geed-3} \\
    &\leq \left( \frac{\epsilon}{4} +5\kappa+\nu_n \right)\sqrt{\dm(\mathcal{V}_{n})} \label{eq.geed-4} \\ & \leq \frac{\epsilon}{2} \sqrt{\dm(\mathcal{V}_n)} \label{eq.gee-13}
\end{align}

This computation can be justified as follows.

\begin{itemize} \item The inequality in (\ref{eq.geed-1}) follows from (\ref{eq.gee-0}). \item (\ref{eq.geed-2}) follows from (\ref{eq.geed-1}) by (\ref{eq.ouch-2}) and (\ref{eq.ouch-1}). \item (\ref{eq.geed-3}) follows from (\ref{eq.geed-2}) since we can use (\ref{eq.chain}) to find: \[ \dm(\mathcal{U}_n) \leq \dm(\mathcal{H}) \leq \left(1-\frac{\kappa^2}{2} \right)^{-m/2} \dm(\mathcal{V}_n) \] \item (\ref{eq.geed-4}) follows from (\ref{eq.geed-3}) by (\ref{eq.sum}). \item (\ref{eq.gee-13}) follows from (\ref{eq.geed-4}) since we chose $\kappa = \epsilon/25$. \end{itemize}

Finally, writing $q_0$ for the projection onto $\mathcal{U}$ we compute: \begin{align} \label{eq.gee-10}  \|\alpha(g)-\beta(g)\|^2_{\mrm{HS}} &= \|\alpha(g)q_0\|^2_{\mrm{HS}} + \sum_{n=0}^{m-1} \|(\alpha(g)-\beta_n(g))p_n\|^2_{\mrm{HS}} \\ &\leq  \dm(\mathcal{U}_0) + \sum_{n=0}^{m-1} \frac{\epsilon^2}{4}  \label{eq.gee-11} \dm(\mathcal{V}_{n})\\
    &\leq \epsilon^2 \dm(\mathcal{H}) \label{eq.gee-12} \end{align} where (\ref{eq.gee-11}) follows from (\ref{eq.gee-10}) by the calculation ending in (\ref{eq.gee-13}) and (\ref{eq.gee-12}) follows from (\ref{eq.gee-11}) by (\ref{eq.chain}). This verifies the required inequality (\ref{eq.want}) and thus completes the proof of Theorem \ref{thm.sofic}. \end{proof}

\bibliographystyle{plain}

\bibliography{CANDIDATE_11.13.bib}

\end{document}